\documentclass[12pt,oneside,english]{amsart}
\usepackage[T1]{fontenc}
\usepackage[latin9]{inputenc}
\usepackage{geometry}
\usepackage{babel}
\usepackage{verbatim}
\usepackage{mathrsfs}
\usepackage{amstext}
\usepackage{amsthm}
\usepackage{amssymb}
\usepackage[unicode=true,pdfusetitle,
 bookmarks=true,bookmarksnumbered=false,bookmarksopen=false,
 breaklinks=false,pdfborder={0 0 1},backref=false,colorlinks=false]
 {hyperref}

\makeatletter
\numberwithin{equation}{section}
\numberwithin{figure}{section}
\theoremstyle{plain}
\newtheorem{thm}{\protect\theoremname}[section]
\theoremstyle{definition}
\newtheorem{defn}[thm]{\protect\definitionname}
\theoremstyle{plain}
\newtheorem{cor}[thm]{\protect\corollaryname}
\theoremstyle{remark}
\newtheorem{rem}[thm]{\protect\remarkname}
\theoremstyle{plain}
\newtheorem{lem}[thm]{\protect\lemmaname}
\theoremstyle{plain}
\newtheorem{prop}[thm]{\protect\propositionname}
\theoremstyle{remark}
\newtheorem*{claim*}{\protect\claimname}
\theoremstyle{definition}
\newtheorem{example}[thm]{\protect\examplename}

\makeatother

\providecommand{\claimname}{Claim}
\providecommand{\corollaryname}{Corollary}
\providecommand{\definitionname}{Definition}
\providecommand{\examplename}{Example}
\providecommand{\lemmaname}{Lemma}
\providecommand{\propositionname}{Proposition}
\providecommand{\remarkname}{Remark}
\providecommand{\theoremname}{Theorem}

\begin{document}
\global\long\def\converge{\underset{n\to\infty}{\longrightarrow}}%

\global\long\def\supp{\text{\textnormal{supp}}}%

\global\long\def\limn{\lim\limits _{n\to\infty}}%

\global\long\def\dimh{\dim_{\text{\textnormal{H}}}}%

\global\long\def\diml{\dim_{\text{\textnormal{L}}}}%

\global\long\def\dima{\dim_{\text{\textnormal{A}}}}%

\global\long\def\rd{\mathbb{R}^{d}}%

\global\long\def\N{\mathbb{N}}%

\global\long\def\Z{\mathbb{Z}}%

\title[microsets of random fractals and Furstenberg's homogeneity]{On microsets, Assouad dimension and lower dimension of random fractals,
and Furstenberg's homogeneity}
\author{Yiftach Dayan}
\begin{abstract}
We study the collection of microsets of randomly constructed fractals,
which in this paper, are referred to as Galton-Watson fractals. This
is a model that generalizes Mandelbrot percolation, where Galton-Watson
trees (whose offspring distribution is not necessarily binomial) are
projected to $\rd$ by a coding map which arises from an iterated
function system (IFS) of similarity maps. We show that for such a
random fractal $E$, whenever the underlying IFS satisfies the open
set condition, almost surely the Assouad dimension of $E$ is the
maximal Hausdorff dimension of a set in $\supp\left(E\right)$, the
lower dimension is the smallest Hausdorff dimension of a set in $\supp\left(E\right)$,
and every value in between is the Hausdorff dimension of some microset
of $E$.

In order to obtain the above, we first analyze the relation between
the collection of microsets of a (deterministic) set, and certain
limits of subtrees of an appropriate coding tree for that set.

The results of this analysis are also applied, with the required adjustments,
to gain some insights on Furstenberg's homogeneity property. Corresponding
to a discussion in \cite{KaenmakiRossi2016}, we define a weaker property
than homogeneity and show that for self-homothetic sets in $\mathbb{R}$
whose Hausdorff dimension is smaller than 1, it is equivalent to the
weak separation condition.
\end{abstract}

\maketitle

\section{Introduction\label{sec:Introduction}}

\subsection{Background}

The notion of a microset was defined by Furstenberg in his paper \cite{Furstenberg2008405}.
In essence, microsets are limits of repeated local magnifications
of a given set. Since the paper \cite{Furstenberg2008405} was published,
several authors used different variations of Furstenberg's definition.
In this paper we use the following:
\begin{defn}
\label{def:microsets}Given a compact set $F\subseteq\mathbb{R}^{d}$,
\begin{enumerate}
\item A non-empty set $A\subset Q$ is called a \emph{miniset} of $F$ if
$A=Q\cap\varphi\left(F\right)$, for some expanding similarity map
$\varphi:\mathbb{R}^{d}\to\mathbb{R}^{d}$.
\item A \emph{microset }of $F$ is any limit of a sequence of minisets of
$F$ in the Hausdorff metric. %
\end{enumerate}
\end{defn}

In the definition above and for the rest of the paper $Q=\left[0,1\right]^{d}$
denotes the closed unit cube in $\mathbb{R}^{d}$, and a \emph{similarity
map} is a function $\varphi:\mathbb{R}^{d}\to\mathbb{R}^{d}$ for
which there exists some $r>0$, which is called the scaling ratio
of $\varphi$, s.t. for every $x,y\in\rd$, $\left\Vert \varphi\left(x\right)-\varphi\left(y\right)\right\Vert =r\cdot\left\Vert x-y\right\Vert $.
In case $r<1$ we call $\varphi$ a \emph{contracting similarity}
and if $r\geq1$ $\varphi$ is called an \emph{expanding similarity}.
Throughout the paper, $\left\Vert \cdot\right\Vert $ denotes the
euclidean norm in $\mathbb{R}^{d}$.

It should be noted that microsets, as defined here, are also called
weak tangents by several authors, a notion which has several variations
as well. 

Recently, there has been an increasing interest in understanding the
collection of microsets of certain sets in $\rd$, and in particular,
their Hausdorff dimensions. It was shown that certain properties of
the collection of microsets of a given compact set have interesting
implications regarding various properties of the set. 

In \cite{Furstenberg2008405}, Furstenberg considered the process
of repeated magnification of a set as a dynamical system called a
CP-process, and used methods from ergodic theory to get results regarding
the Hausdorff dimension of the image of the set under linear transformations,
whenever the collection of microsets satisfies certain conditions
(see section \ref{sec:Fursternberg} for more details). As another
example, certain properties of the collection of microsets of a given
compact set provide interesting information regarding the intersection
of the set with winning sets for a variation of Schmidt's game called
the hyperplane absolute game (\cite{Broderick2012319}, \cite{Das2016}). 

It is also known that the Hausdorff dimension of microsets is linked
to two different notions of dimension - the \emph{Assouad dimension}
and the \emph{lower dimension} (denoted by $\dima$ and $\diml$).
These provide a coarse and localized information regarding the density
of a set in asymptotically small scales. The Assouad dimension describes
the densest part of the set while the lower dimension describes the
sparsest part of the set. These notions are also gaining popularity
recently, and were studied in several recent papers.

Throughout this paper we assume some basic familiarity with Hausdorff
dimension (denoted by $\dimh)$ and its basic properties. For an introduction
the reader is referred to \cite{falconer2013fractal}.

The relation between the Hausdorff dimension of microsets and the
Assouad and lower dimensions is given in the theorem below. Since
for the lower dimension it is important to only consider microsets
which intersect the interior of $Q$ (this has no influence for the
Assouad dimension), we denote the collection of those microsets of
a given set $F\subseteq\rd$ by $\mathcal{W}_{F}$. 
\begin{thm}
\label{thm:Assouad as max}Given a non-empty compact set $F\subseteq\mathbb{R}^{d}$,
the following hold:
\begin{enumerate}
\item $\dima\left(F\right)=\max\left\{ \dimh\left(A\right):\,A\in\mathcal{W}_{F}\right\} $
\item $\diml\left(F\right)=\min\left\{ \dimh\left(A\right):\,A\in\mathcal{W}_{F}\right\} $
\end{enumerate}
\end{thm}

The first point of the above theorem was proved in \cite{Kaenmaki2018rigidity},
based on ideas of \cite[section 2.4]{bishop2016fractals} and \cite{Furstenberg2008405}.
The second point is from \cite{Fraser2019}. See also \cite[section 5]{Fraser2020}. 

The Assouad and lower dimensions will be formally defined in subsection
\ref{subsec:Assouad and lower dimensions}. However, while their formal
definition is rather technical, Theorem \ref{thm:Assouad as max}
provides a more intuitive equivalent definition for compact sets.

Since addressing every microset of a given set is often impractical,
it is common to study the collection of microsets of a given set,
and specifically, their Hausdorff dimension, by examining appropriate
coding trees for the given set. This is the approach taken in this
paper, where we concentrate here in coding trees arising from iterated
function systems. This includes coding by digital expansion for example.

The main goal of this paper is to study the collection of microsets
and the Assouad and lower dimensions of certain randomly constructed
fractals, which we refer to as \emph{Galton-Watson fractals (GWF)}
due to their relation to Galton-Watson processes. The analysis presented
in this paper joins a few recent papers (\cite{Fraser2018assouad,Troscheit2017dimensions,Troscheit2020quasi,HowroydYu2019})
who studied the Assouad dimension of sets generated by some related
probabilistic models, as well as certain related notions (quasi-Assouad
dimension, and Assouad spectrum). These models include Mandelbrot
percolation \cite{Fraser2018assouad}, graphs of (fractional) Brownian
motions \cite{HowroydYu2019}, 1-variable random iterated function
systems \cite{Fraser2018assouad}, and the more general random graph
directed attractors \cite{Troscheit2017dimensions,Troscheit2020quasi}.

In general, for a set $E$ generated by a probabilistic model which
assumes some statistical self similarity, it is expected that the
Assouad dimension of $E$ will be ``as large as possible''. In Theorem
\ref{thm:Hausdorff dimensions of microsets of GWF} we make this precise
for GWF. We show that under certain conditions, for a typical GWF
$E$, the Assouad dimension of $E$ is equal to the maximal Hausdorff
dimension of sets in the support of $E$. Here, the support of $E$
(which is also denoted by $\supp\left(E\right)$) means the support
of the probability distribution of $E$ in a suitable space of compact
sets in $\rd$. Analogously, we show that typically, the lower dimension
of $E$ is the minimal Hausdorff dimension of sets in $\supp\left(E\right)$.
We also show that both quantities, $\dimh\left(E\right)$ and $\diml\left(E\right)$,
may be very easily calculated. 

We show the above by analyzing the collection of microsets of a typical
realization of $E$, studying the relation between the collection
$\mathcal{W}_{E}$ and sets in the support of $E$. As a result we
also obtain that for a typical realization of $E$, every number in
the interval $\left[\diml\left(E\right),\,\dima\left(E\right)\right]$,
is the Hausdorff dimension of some microset in $\mathcal{W}_{E}$.

This is a result of a general analysis of deterministic sets, relating
certain properties of their coding trees to their microsets. We also
apply these results with the required adjustments, to gain some insights
on Furstenberg's homogeneity property. This is done in Section \ref{sec:Fursternberg}.

\subsection{Some preparations}

We now make some necessary preparations for stating the main results
of this paper. All notions presented here will be discussed in further
detail in Section \ref{sec:Preliminaries}.

\subsubsection{Coding trees}

Given a finite iterated function system (IFS) $\Phi=\left\{ \varphi_{i}\right\} _{i\in\Lambda}$
of contracting similarity maps of $\rd$, we denote the attractor
of $\Phi$ by $K_{\Phi}$ or just by $K$ whenever the IFS is clear.
$\Phi$ defines a continuous and surjective coding map $\gamma_{\Phi}:\Lambda^{\mathbb{N}}\to K$,
given by 
\[
\forall i=\left(i_{1},i_{2},...\right)\in\Lambda^{\mathbb{N}},\,\,\,\,\,\gamma_{\Phi}\left(i\right)=\lim_{n\rightarrow\infty}\varphi_{i_{1}}\circ\cdots\circ\varphi_{i_{n}}(0).
\]
For every compact subset $F\subseteq K$, there is a compact set $P\subseteq\Lambda^{\mathbb{N}}$
which codes $F$, i.e., such that $\gamma_{\Phi}\left(P\right)=F$. 

The statements of the main results below involve the following separation
conditions and their respective abbreviations: open set condition
(OSC), weak separation condition (WSC), strong separation condition
(SSC). These will be defined and discussed in subsection \ref{subsec:Iterated-function-systems}.

Given a finite set $\mathbb{A}$, we denote $\mathbb{A}^{*}=\bigcup\limits _{n=0}^{\infty}\mathbb{A}^{n}$
where $\mathbb{A}^{0}=\left\{ \emptyset\right\} $. A \emph{tree}
with alphabet $\mathbb{A}$ is a set $T\subseteq\mathbb{A}^{*}$,
such that $\emptyset\in T$ (which is considered as the root of the
tree), and for every word $a\in T$, every prefix of $a$ is also
a member of $T$. The \emph{boundary} of $T$ is a compact subset
of $\mathbb{A}^{\mathbb{N}}$, defined by $\partial T=\left\{ i_{1}i_{2}\dots\hphantom{}\in\mathbb{A}^{\mathbb{N}}:\,i_{1}\dots i_{n}\in T,\,\forall n\in\mathbb{N}\right\} $.
Given $\Phi$ and $F\subseteq K_{\Phi}$, an infinite tree $T$ with
alphabet $\Lambda$ such that $\gamma_{\Phi}\left(\partial T\right)=F$
is called a \emph{coding tree} for $F$. We further denote $\Gamma_{\Phi}\left(T\right)=\gamma_{\Phi}\left(\partial T\right)$. 

We will often be interested in trees that have no leaves, which means
that $\forall i_{1}\dots i_{n}\in T$ there is some $j\in\mathbb{A}$
such that $i_{1}\dots i_{n}j\in T$, or in words, this means that
every node of the tree has at least one child. For a nonempty finite
set $\mathbb{A}$, the set of trees in the alphabet $\mathbb{A}$
is denoted by $\mathscr{T}_{\mathbb{A}}$, and its subset of trees
with no leaves is denoted by $\mathscr{T}_{\mathbb{A}}^{\prime}$.

\subsubsection{\label{subsec:Galton-Watson-fractals}Galton-Watson fractals}

Given an IFS $\Phi=\left\{ \varphi_{i}\right\} _{i\in\Lambda}$ as
above, let $W$ be a random variable taking values in the finite set
$2^{\Lambda}$. A \emph{Galton-Watson tree} \emph{(GWT)} is a random
tree constructed by iteratively choosing at random the children of
each node of the tree as realizations of independent copies of $W$
concatenated to the current node, starting from the root, namely $\emptyset$.
Whenever $\mathbb{P}\left[W=\emptyset\right]>0$, it may happen that
in some generation of the tree there will be no surviving elements.
In that case, the resulting tree is finite, and we say that the event
of \emph{extinction} occurred. It is well known that unless almost
surely $\left|W\right|=1$, $\mathbb{E}\left(\left|W\right|\right)\leq1\iff\text{extinction occurs a.s.}$
(see \cite[Chapter 5]{Lyons2016}). The case $\mathbb{E}\left(\left|W\right|\right)>1$
is called \emph{supercritical} and in this article, we shall always
assume this is the case.

Conditioned on the event of non-extinction, the projection of the
Galton-Watson tree to $\rd$ using $\Gamma_{\Phi}$ yields a random
(non-empty) compact subset of the attractor $K$, which we call a
\emph{Galton-Watson fractal} \emph{(GWF)}. For convenience, in this
paper, we take the condition of non-extinction into the definition
of a GWF, and so the probability that a GWF is empty is always 0.

An important special case of GWF is a well known model called \emph{fractal
percolation} (AKA \emph{Mandelbrot percolation}). Given two parameters:
$2\leq b\in\mathbb{Z}$, and $p\in\left(0,1\right)$, fractal percolation
in $\rd$ may be defined as the GWF corresponding to the IFS $\Phi=\left\{ \varphi_{i}:\,x\mapsto\dfrac{1}{b}x+\dfrac{i}{b}\right\} _{i\in\Lambda}$
where $\Lambda=\left\{ 0,\dots,b-1\right\} ^{d}$, and $W$ has a
binomial distribution, i.e., every element of $\Lambda$ is contained
in $W$ with probability $p$, and independently of all other members
of $\Lambda$. Geometrically, this model may also be described by
starting from the unit cube, and successively partitioning the surviving
cubes into $b^{d}$ subcubes of the same dimensions, where each of
them either survives with probability $p$ or is discarded with probability
$1-p$, independently of the others. Denoting by $E_{n}$ the union
of all surviving cubes in the $n^{\text{th}}$ generation (so that
$E_{0}=Q$), the resulting set is $E=\underset{n\in\mathbb{N}}{\bigcap}E_{n}$,
which may also be empty if at some point of the process no subcube
survives (AKA extinction). Hence the corresponding GWF has the distribution
of $E$ conditioned on non-extinction.

\subsection{Main results}

\subsubsection{Microsets via limits of descendant trees}

Given a tree $T\in\mathscr{T}_{\mathbb{A}}$, for any $v\in T$, one
may consider the tree growing from the node $v$ (thinking of $v$
as the root). We call this tree the descendants tree of $v$. It is
denoted by $T^{v}$ and formally defined by
\[
T^{v}=\left\{ w\in\mathbb{A}^{*}:\,vw\in T\right\} \in\mathscr{T}_{\mathbb{A}}.
\]
Given an IFS $\Phi=\left\{ \varphi_{i}\right\} _{i\in\Lambda}$ and
$T\in\mathscr{T}_{\Lambda}^{\prime}$, we denote $\mathcal{S}_{\Phi}^{T}=\left\{ \Gamma_{\Phi}\left(T^{v}\right):\,v\in T\right\} $,
and denote by $\overline{\mathcal{S}_{\Phi}^{T}}$ the closure of
$\mathcal{S}_{\Phi}^{T}$ in the Hausdorff metric. 

The following two theorems discuss the relation between microsets
of a given compact set in $\mathbb{R}^{d}$ and $\overline{\mathcal{S}_{\Phi}^{T}}$
for an appropriate coding tree $T$.
\begin{thm}
\label{thm:microset is a union of minisets}Let $F\subseteq\rd$ be
a non-empty compact set. Given a similarity IFS $\Phi=\left\{ \varphi_{i}\right\} _{i\in\Lambda}$
which satisfies WSC and a tree $T\in$$\mathbb{\mathscr{T}}_{\Lambda}^{\prime}$
such that $F=\Gamma_{\Phi}\left(T\right)$, there exists some $n\geq1$
which depends only on $\Phi$, such that for every microset $M$ of
$F$, there are $A_{1},...,A_{n}\in\overline{\mathcal{S}_{\Phi}^{T}}$
and expanding similarity maps $\psi_{1},...,\psi_{n}$, such that
\begin{equation}
Q^{\circ}\cap\left(\bigcup_{j=1}^{n}\psi_{j}A_{j}\right)\subseteq M\subseteq Q\cap\left(\bigcup_{j=1}^{n}\psi_{j}A_{j}\right).\label{eq:microsets are minisets}
\end{equation}
Moreover, if $\Phi$ satisfies the SSC, then $n=1$.
\end{thm}

The theorem above says that except for a possible difference on the
boundary of $Q$, every microset of $F$ may be represented as a finite
union of minisets of sets in $\overline{\mathcal{S}_{\Phi}^{T}}$,
for any coding tree $T$ without leaves, as long as the underlying
IFS satisfies the WSC.%
{} On the other hand, the following holds:
\begin{thm}
\label{thm:minisets are microsets}Let $F\subseteq\rd$ be a non-empty
compact set. Given a similarity IFS $\Phi=\left\{ \varphi_{i}\right\} _{i\in\Lambda}$
and a tree $T\in$$\mathbb{\mathscr{T}}_{\Lambda}^{\prime}$ such
that $F=\Gamma_{\Phi}\left(T\right)$, every miniset $M$ of a member
of $\overline{\mathcal{S}_{\Phi}^{T}}$ is contained in a microset
of $F$. 

Moreover, if $M$ has the form $M=Q\cap\psi\left(A\right)$ for some
$A\in\overline{\mathcal{S}_{\Phi}^{T}}$ and an expanding similarity
map $\psi$, and assuming that $\Phi$ satisfies the OSC with an OSC
set $U$ such that $\psi^{-1}\left(Q\right)\subseteq U$, then $M$
is a microset of $F$.%
\end{thm}

The notion of an OSC set is a part of the definition of the OSC, which
is given in Definition \ref{def:OSC and SSC}. 

The difference between the sets in Equation (\ref{eq:microsets are minisets})
is contained in the boundary of $Q$. Since $\dimh\left(\partial Q\right)=d-1$,
whenever $\dimh\left(M\right)>d-1$, this difference has no influence
on the Hausdorff dimension of the sets. Hence the following is an
immediate corollary of Theorem \ref{thm:microset is a union of minisets}.
\begin{cor}
\label{cor:microsets with dim>d-1}Under the assumptions of Theorem
\ref{thm:microset is a union of minisets}, for every microset $M$
of $F$ with $\dimh\left(M\right)>d-1$, there is a miniset $A$ of
some set in $\overline{\mathcal{S}_{\Phi}^{T}}$, with $\dimh\left(A\right)=\dimh\left(M\right)$.
\end{cor}

In particular, in the special case where $d=1$, the above holds for
every $M\in\mathcal{W}_{F}$ (this includes the case $\dimh\left(M\right)=0$
since there are minisets whose Hausdorff dimension is 0). Note that
$\dimh\left(M\right)>d-1$ implies that $M\cap Q^{\circ}\neq\emptyset$
and so $M\in\mathcal{W}_{F}$, and $A\cap Q^{\circ}\neq\emptyset$
as well.
\begin{rem}
If one wishes to move from minisets of elements of $\overline{\mathcal{S}_{\Phi}^{T}}$
to just elements of $\overline{\mathcal{S}_{\Phi}^{T}}$, one needs
to take the closure of the set of Hausdorff dimensions, that is 
\[
\left\{ \dimh\left(A\right):\,A\in\mathcal{W}_{F}\right\} \cap\left(d-1,d\right]\subseteq\overline{\left\{ \dimh\left(S\right):\,S\in\overline{\mathcal{S}_{\Phi}^{T}}\right\} }.
\]
For more details see Subsection \ref{subsec:Hausdorff dim of branch sets vs their minisets},
and in particular Proposition \ref{prop:minisets of branch sets in int(Q) vs branch sets dim>d-1}.
\end{rem}

If the underlying IFS satisfies the strong separation condition, the
converse direction of Corollary \ref{cor:microsets with dim>d-1}
holds as well. See Corollary \ref{cor:dim of minisets contained in dim of microsets under SSC}.

\subsubsection{Consequences regarding Assouad and lower dimensions }

After the relation between $\mathcal{W}_{F}$ and $\overline{\mathcal{S}_{\Phi}^{T}}$
has been established, the following may be deduced regarding the lower
and Assouad dimensions.
\begin{thm}
\label{thm:Assuad and lower dimensions as branching sets}Let $F\subseteq\mathbb{R}^{d}$
be a non-empty compact set, and let $\Phi=\left\{ \varphi_{i}\right\} _{i\in\Lambda}$
be a similarity IFS which satisfies the WSC, such that $F=\gamma_{\Phi}\left(\partial T\right)$
for some tree $T\in$$\mathbb{\mathscr{T}}_{\Lambda}^{\prime}$.
\begin{enumerate}
\item \label{enu:dim_A}$\dima\left(F\right)=\max\left\{ \dimh\left(S\right):\,S\in\overline{\mathcal{S}_{\Phi}^{T}}\right\} $.
\item \label{enu:dim_L_lower_bound} $\diml\left(F\right)\geq\inf\left\{ \text{\ensuremath{\dimh\left(S\right):\,S\in}}\overline{\mathcal{S}_{\Phi}^{T}}\right\} $.
Moreover, $\exists S\in\overline{\mathcal{S}_{\Phi}^{T}}$ with $\dimh\left(S\right)\leq\diml\left(F\right)$.
\item \label{enu:dim_L upper bound}$\diml\left(F\right)\leq\inf\left\{ \dimh\left(\psi\left(S\right)\cap Q\right):\,\begin{array}{l}
\text{\ensuremath{S\in}}\overline{\mathcal{S}_{\Phi}^{T}},\,\\
\psi\text{ is a similarity function,}\\
\psi^{-1}\left(Q\right)\subseteq U\text{ for an OSC set \ensuremath{U},}\\
\text{and }\text{\ensuremath{\psi\left(S\right)\cap Q^{\circ}\neq\emptyset}}
\end{array}\right\} $, if $\Phi$ satisfies the OSC.
\item \label{enu:dim_L SSC}$\diml\left(F\right)=\min\left\{ \text{\ensuremath{\dimh\left(S\right):\,S\in}}\overline{\mathcal{S}_{\Phi}^{T}}\right\} $,
if $\Phi$ satisfies the SSC.
\end{enumerate}
\end{thm}

The reader should note that for the special case where all the contraction
ratios of the maps of $\Phi$ are equal, (\ref{enu:dim_A}) was proved
in \cite{Kaenmaki2018rigidity}, but with a $\sup$ instead of a $\max$.\footnote{In fact it was proved for a more general construction called Moran
construction, however some restriction on the diameters of the sets
in this construction implies that for similarity IFSs the result may
only be applied when all the contraction ratios are equal.} 

\subsubsection{Microsets of Galton-Watson fractals}

Let $E$ be a Galton-Watson fractal w.r.t. a similarity IFS $\Phi=\left\{ \varphi_{i}\right\} _{i\in\Lambda}$.
Then $E$ may be considered as a random variable which takes values
in the space of compact subsets of $K_{\Phi}$. We denote by $\supp\left(E\right)$
the support of the distribution of $E$, which is a Borel probability
measure on the space of compact subsets of $K_{\Phi}$, w.r.t. the
topology induced by the Hausdorff metric. Let $T$ be the corresponding
Galton-Watson tree (so that $E$ is defined by $\Gamma_{\Phi}\left(T\right)$
conditioned on non-extinction), and denote by $T^{\prime}$ the reduced
tree without leaves (whenever $T$ is infinite). That is, $T^{\prime}$
is the set of all nodes of $T$ that have an infinite line of descendants,
hence $T^{\prime}\in\mathscr{T}_{\Lambda}^{\prime}$. We show that
a.s. conditioned on non-extinction, $\overline{\mathcal{S}_{\Phi}^{T^{\prime}}}=\supp\left(E\right)$.
Thus, applying Theorems \ref{thm:microset is a union of minisets},
\ref{thm:minisets are microsets}, we obtain a strong relation between
$\supp\left(E\right)$ and $\mathcal{W}_{E}$ (Theorems \ref{thm:microsets as minisets for GWF},
\ref{thm:minisets as microsets for GWF}). This is done in Subsection
\ref{subsec:Microsets-and-dimensions-of GWF}. 

As an example, in the simple case where $E$ is a realization of a
supercritical Mandelbrot percolation in $\rd$, its probability distribution
is fully supported on the space of compact subsets of $Q$. Hence,
we obtain that a.s. conditioned on non-extinction, every closed subset
of the unit cube is a microset of $E$. In particular, this implies
that $\left\{ \dimh\left(F\right):\,F\in\mathcal{W}_{E}\right\} =\left[0,d\right]$,
and therefore $\diml\left(E\right)=0$ and $\dima\left(E\right)=d$
(the latter was already shown in \cite{Fraser2018assouad}).

In the general case, we obtain the following theorem:
\begin{thm}
\label{thm:Hausdorff dimensions of microsets of GWF}Let $E$ be a
Galton-Watson fractal constructed with respect to a similarity IFS
$\Phi=\left\{ \varphi_{i}\right\} _{i\in\Lambda}$, which satisfies
the OSC, and an offspring distribution $W$. Denote,
\[
\begin{array}{c}
m_{W}=\min\limits _{A\in\supp\left(W\right)}\dimh\left(K_{A}\right)\\
M_{W}=\max\limits _{A\in\supp\left(W\right)}\dimh\left(K_{A}\right)
\end{array},
\]
where in case $\emptyset\in\supp\left(W\right)$, we set $m_{W}=0$. 

Then 
\[
\left\{ \dimh\left(F\right):\,F\in\supp\left(E\right)\right\} =\left[m_{W},M_{W}\right],
\]
and almost surely,
\[
\left\{ \dimh\left(F\right):\,F\in\mathcal{W}_{E}\right\} =\left[m_{W},M_{W}\right].
\]
In particular, almost surely
\[
\diml\left(E\right)=m_{W},\,\dima\left(E\right)=M_{W}.
\]
\end{thm}

In the statement of the theorem above, $\supp\left(W\right)$ denotes
the support of the probability distribution of $W$ on the finite
set $2^{\Lambda}$. Hence, $m_{W}$ and $M_{W}$ may be very easily
calculated whenever the offspring distribution $W$ is given. However,
they have another meaning as well, namely, they are the minimal and
maximal Hausdorff dimensions of sets in $\supp\left(E\right)$.
\begin{rem}
Another notion of dimension which is related to the lower dimension
is the \emph{modified lower dimension}. For any $A\subseteq\rd$,
the modified lower dimension of $A$ is defined by 
\[
\dim_{\text{ML}}\left(A\right)=\sup\left\{ \diml\left(B\right):\,B\subseteq A\right\} .
\]
For more details the reader is referred to \cite{Fraser2020}. We
remark here that for any GWF $E\subseteq\rd$, which is constructed
with respect to a similarity IFS that satisfies the OSC, almost surely,
\[
\dim_{\text{ML}}\left(E\right)=\dimh\left(E\right).
\]
This follows from Theorem 1.8 in \cite{dayan_2021}, which asserts
that $E$ a.s. contains Ahlfors-regular subsets whose Hausdorff dimension
(which is equal to their lower dimension) is arbitrarily close to
$\dimh\left(E\right)$.
\end{rem}

\subsubsection{Furstenberg's homogeneity}

So far we have only considered minisets and microsets as defined
in Definition \ref{def:microsets}. However, this definition is slightly
different than Furstenberg's definition from \cite{Furstenberg2008405},
which allows only homotheties rather than general similarity maps,
and is closed under taking compact subsets. This is formally defined
in Definition \ref{def:F-microsets}, using the phrases \emph{F-microsets}
and \emph{F-minisets}. 

Allowing compact subsets in the definition of F-minisets, makes the
relation between the collection of F-microsets of a given compact
set and $\overline{\mathcal{S}_{\Phi}^{T}}$, for an appropriate coding
tree $T$ and an IFS of homotheties $\Phi$, nicer than the one established
in Theorems \ref{thm:microset is a union of minisets}, \ref{thm:minisets are microsets}.
Namely, for a non-empty compact set $F\subseteq\rd$ , an IFS of homotheties
$\Phi=\left\{ \varphi_{i}\right\} _{i\in\Lambda}$, and $T\in$$\mathbb{\mathscr{T}}_{\Lambda}^{\prime}$
a coding tree for $F$, we obtain that every F-miniset of a member
of $\overline{\mathcal{S}_{\Phi}^{T}}$ is an F-microset of $F$,
and on the other hand, whenever $\Phi$ satisfies the WSC, every F-microset
of $F$ is a (bounded) finite union of F-minisets of members of $\overline{\mathcal{S}_{\Phi}^{T}}$.
This is the content of Theorems \ref{thm:Furstenberg microset is a union of minisets}
and \ref{thm:F-minisets are F-microsets} in Section \ref{sec:Fursternberg}.

If $F$ is the attractor of $\Phi$, so that we may take $T=\Lambda^{*}$
and hence $\overline{\mathcal{S}_{\Phi}^{T}}=\left\{ F\right\} $,
we obtain from the above that if $\Phi$ satisfies the WSC, then every
F-microset of $F$ is a finite union of F-minisets of $F$. We view
this property as a weaker form of Furstenberg's homogeneity property
(where every F-microset is an F-miniset), and show that for self-homothetic
sets in $\mathbb{R}$ whose Hausdorff dimension is smaller than 1,
this property is equivalent to WSC. 

This corresponds to a discussion from \cite{KaenmakiRossi2016} where
it was shown that for self-homothetic subsets of $\mathbb{R}$ whose
Hausdorff dimension is smaller than 1, homogeneity implies WSC, but
included a counter example for the converse implication as well, showing
that homogeneity is not equivalent to WSC for that class of sets.

\subsection{Structure of the paper}

In Section \ref{sec:Preliminaries} some relevant preparations are
made. First we introduce some basic definitions regarding symbolic
spaces and trees. Then we recall some properties of the Hausdorff
metric, which will be useful in what follows. Then IFSs and their
attractors are discussed, and some technical properties of the coding
map are established. Finally the lower and Assouad dimensions are
formally defined. 

In Section \ref{sec:Microsets-via-coding} the relation between the
collection of microsets and the collection $\overline{\mathcal{S}_{\Phi}^{T}}$
is discussed. In particular, this section contains the proofs of Theorems
\ref{thm:microset is a union of minisets}, \ref{thm:minisets are microsets},
and \ref{thm:Assuad and lower dimensions as branching sets}, as well
as some of their corollaries. 

Section \ref{sec:Galton-Watson-fractals} deals with Galton-Watson
fractals. After providing a formal definition, and discussing some
relevant properties regarding the reduced trees of Galton-Watson trees,
and the support of Galton-Watson fractals, we apply the results of
Section \ref{sec:Microsets-via-coding} and obtain some results regarding
the microsets of GWF. In particular Theorem \ref{thm:Hausdorff dimensions of microsets of GWF}
is proved. 

In Section \ref{sec:Fursternberg}, we recall Furstenberg's original
definition of microsets, and make some adjustments of the results
from previous sections to Furstenberg's definition. Furstenberg's
homogeneity property is also discussed. In view of a discussion from
\cite{KaenmakiRossi2016}, a weaker property is suggested for consideration,
which still implies dimension conservation (as defined in \cite{Furstenberg2008405}),
but is equivalent to the WSC for compact subsets of $\mathbb{R}$
whose Hausdorff dimension is smaller than 1.

\section{Preliminaries\label{sec:Preliminaries}}

\subsection{Symbols and trees\label{subsec:Symbols-and-trees}}

Given a finite set $\mathbb{A}$, as mentioned in Section \ref{sec:Introduction},
we denote by $\mathbb{A}^{*}$ the set of all finite words in the
alphabet $\mathbb{A}$, i.e. $\mathbb{A}^{*}=\bigcup\limits _{n=0}^{\infty}\mathbb{A}^{n}$
where $\mathbb{A}^{0}=\left\{ \emptyset\right\} $. Given a word $i\in\mathbb{A}^{*}$,
we use subscript indexing to denote the letters comprising $i$, so
that $i=i_{1}...i_{n}$ where $i_{k}\in\mathbb{A}$ for $k=1,...,n$.
$\mathbb{A}^{*}$ is considered as a semigroup with the concatenation
operation $\left(i_{1}...i_{n}\right)\cdot\left(j_{1}...j_{m}\right)=\left(i_{1}...i_{n},j_{1}...j_{m}\right)$
and with $\emptyset$ the identity element. The dot notation will
usually be omitted so that the concatenation of two words $i,j\in\mathbb{A}^{*}$
will be denoted simply by $ij$. We will also consider the action
of $\mathbb{A}^{*}$ on $\mathbb{A}^{\mathbb{N}}$ by concatenations
denoted in the same way. 

We put a partial order on $\mathbb{A}^{*}\cup\mathbb{A}^{\mathbb{N}}$
by defining $\forall i\in\mathbb{A}^{*},\,\forall j\in\mathbb{A}^{*}\cup\mathbb{A}^{\mathbb{N}},\,i\leq j$
iff $\exists k\in\mathbb{A}^{*}\cup\mathbb{A}^{\mathbb{N}}$ with
$ik=j$, that is to say $i\leq j$ iff $i$ is a prefix of $j$. 

Given any $i\in\mathbb{A}^{*}$ we shall denote the length of $i$
by $\left|i\right|=n$, where $n$ is the unique integer with the
property $i\in\mathbb{A}^{n}$. The cylinder set in $\mathbb{A}^{\mathbb{N}}$
corresponding to $i$ is defined by $\left[i\right]=\left\{ j\in\mathbb{A}^{\mathbb{N}}:\,i<j\right\} $.
A finite set $\Pi\subset\mathbb{A}^{*}$ is called a \emph{section}
if $\bigcup\limits _{i\in\Pi}\left[i\right]=\mathbb{A}^{\mathbb{N}}$
and the union is a disjoint union.
\begin{defn}
\label{def:tree}A subset $T\subseteq\mathbb{A}^{*}$\emph{ }will
be called a\emph{ tree} with alphabet $\mathbb{A}$ if $\emptyset\in T$,
and for every $a\in T$, $\forall b\in\mathbb{A}^{*},\,b\leq a\implies b\in T$. 
\end{defn}

We shall denote $T_{n}=T\cap\mathbb{A}^{n}$ for every $n\geq1$,
and $T_{0}=\left\{ \emptyset\right\} $ so that $T=\bigcup\limits _{n\geq0}T_{n}$.
Moreover, for any section $\Pi\subset\mathbb{A}^{*}$ we denote $T_{\Pi}=T\cap\Pi$.
For every $a\in T$, we denote $W_{T}\left(a\right)=\left\{ i\in\mathbb{A}:\,ai\in T\right\} $.
The \emph{boundary} of a tree $T$ is denoted by $\partial T$ and
is defined by 
\[
\partial T=\left\{ a\in\mathbb{A}^{\mathbb{N}}:\,\forall n\in\mathbb{N},\,a_{1}...a_{n}\in T_{n}\right\} .
\]
A \emph{subtree} of $T$ is any tree $\tilde{T}$ in the alphabet
$\mathbb{A}$, such that $\tilde{T}\subseteq T$. 

The set of all trees with alphabet $\mathbb{A}$ will be denoted by
$\mathscr{\ensuremath{T}}_{\mathbb{A}}\subset2^{\mathbb{A}^{*}}$.
We define a metric on $\mathscr{\ensuremath{T}}_{\mathbb{A}}$ by
\[
d_{\mathbb{A}}\left(T,S\right)=2^{-\min\left\{ n:\,T_{n}\neq S_{n}\right\} }.
\]
Endowed with the metric $d_{\mathbb{A}}$, $\mathscr{\ensuremath{T}}_{\mathbb{A}}$
is a compact, separable, metric space. 

Given a finite tree $L\in\mathscr{\ensuremath{T}}_{\mathbb{A}}$,
let $\left[L\right]\subset\mathscr{\ensuremath{T}}_{\mathbb{A}}$
be defined by 
\[
\left[L\right]=\left\{ S\in\mathscr{\ensuremath{T}}_{\mathbb{A}}:\,\forall n\in\mathbb{N},\,L_{n}\neq\emptyset\implies S_{n}=L_{n}\right\} .
\]
These sets form a (countable) basis for the topology of $\mathscr{T}_{\mathbb{A}}$
and generate the Borel $\sigma$-algebra on $\mathscr{T}_{\mathbb{A}}$.

We continue with a few more definitions which will come in handy in
what follows. Given a tree $T\in\mathscr{T}_{\mathbb{A}}$ and a node
$a\in T$, we denote 
\[
T^{a}=\left\{ j\in\mathbb{A}^{*}:\,aj\in T\right\} \in\mathscr{T}_{\mathbb{A}},
\]
the\emph{ descendants tree of $a$}. The \emph{height} of a tree $T\subseteq\mathbb{A}^{*}$
is defined by 
\[
\text{height}\left(T\right)=\sup\left\{ n\in\mathbb{N}\setminus\left\{ 0\right\} :\,T_{n}\neq\emptyset\right\} ,
\]
and takes values in $\mathbb{N}\cup\left\{ \infty\right\} $. A basic
observation in this context is that $\forall T\in\mathscr{T}_{\mathbb{A}}$,
$\text{height}\left(T\right)=\infty\iff\partial T\neq\emptyset$.

In what follows, it will often make sense to restrict the discussion
to certain subspaces of $\mathscr{T}_{\mathbb{A}}$, namely, the subspace
of all infinite trees in the alphabet $\mathbb{A}$, which is defined
by 
\[
\mathscr{T}_{\mathbb{A}}^{\infty}=\left\{ T\in\mathscr{T}_{\mathbb{A}}:\,\text{height}\left(T\right)=\infty\right\} ,
\]
and the subspace of all trees in the alphabet $\mathbb{A}$ that have
no leaves, which is defined by
\[
\mathbb{\mathscr{T}}_{\mathbb{A}}^{\prime}=\left\{ T\in\mathscr{T}_{\mathbb{A}}:\,\forall v\in T,\,\exists w\in T,\,v<w\right\} .
\]

Clearly, $\mathbb{\mathscr{T}}_{\mathbb{A}}^{\prime}\subseteq\mathscr{T}_{\mathbb{A}}^{\infty}\subseteq\mathscr{T}_{\mathbb{A}}$,
where $\mathscr{T}_{\mathbb{A}}^{\infty},\,\mathbb{\mathscr{T}}_{\mathbb{A}}^{\prime}$
are both closed subsets of $\mathscr{T}_{\mathbb{A}}$ and hence,
equipped with the induced topology, are also compact. There is a natural
surjection $\mathscr{T}_{\mathbb{A}}^{\infty}\to\mathbb{\mathscr{T}}_{\mathbb{A}}^{\prime}$,
which assigns to every infinite tree $T\in\mathscr{T}_{\mathbb{A}}^{\infty}$
, the \emph{reduced tree}\footnote{This terminology is taken from \cite{Lyons2016}.}
with no leaves $T^{\prime}\in\mathbb{\mathscr{T}}_{\mathbb{A}}^{\prime}$,
which is defined by
\[
T^{\prime}=\left\{ v\in T:\,T^{v}\in\mathscr{T}_{\mathbb{A}}^{\infty}\right\} .
\]

\subsection{Hausdorff metric}

For any $x\in\mathbb{R}^{d}$ and $r>0$, we denote by $B_{r}\left(x\right)$
the open ball around $x$ of radius $r$, that is:
\[
B_{r}\left(x\right)=\left\{ y\in\mathbb{R}^{d}:\,\left\Vert x-y\right\Vert <r\right\} .
\]
Given any set $A\subseteq\mathbb{R}^{d}$ and $\varepsilon>0$, the
(closed) $\varepsilon$-neighborhood of $A$ is denoted by 
\[
A^{\left(\varepsilon\right)}=\left\{ x\in\mathbb{R}^{d}:\,\exists a\in A,\,\left\Vert x-a\right\Vert \leq\varepsilon\right\} .
\]
The interior of a set $A$ which is contained in some topological
space is denoted by $A^{\circ}$.

We now recall the definition of the Hausdorff metric and some of its
basic properties. 
\begin{defn}
Given two sets $A,B\subseteq\mathbb{R}^{d}$, we denote by $d_{H}\left(A,B\right)$
the Hausdorff distance between $A$ and $B$, given by
\[
d_{H}\left(A,B\right)=\inf\left\{ \varepsilon:\,A\subseteq B^{\left(\varepsilon\right)}\text{ and }B\subseteq A^{\left(\varepsilon\right)}\right\} .
\]

Given any non-empty compact set $F\subset\mathbb{R}^{d}$, we denote
by $\Omega_{F}$ the set of all non-empty compact subsets of $F$.
Restricted to $\Omega_{F}$, $d_{H}$ is a metric. Throughout this
paper $\Omega_{F}$ is always endowed with the topology induced by
the Hausdorff metric, which makes $\left(\Omega_{F},\,d_{H}\right)$
a complete, metric, compact space. However, in what follows, the particular
space $\Omega_{F}$ will not always be explicitly stated and will
only be present in the background, assuming that the set $F$ contains
all the relevant sets.
\end{defn}

We now state some basic lemmas regarding convergence of sequences
in the Hausdorff metric, in an appropriate space. The proofs are standard
and are left as exercises for the reader.
\begin{lem}
\label{lem:convergence of intersections of two sequences}Let $A_{n,}\,B_{n}$
be two sequences of nonempty compact sets in $\mathbb{R}^{d}$. Suppose
that $A_{n}\underset{n\to\infty}{\longrightarrow}A,\,\text{and }B_{n}\underset{n\to\infty}{\longrightarrow}B$
for some compact sets $A,B\subset\mathbb{R}^{d}$. Then every accumulation
point of the sequence $A_{n}\cap B_{n}$ is a subset of $A\cap B$.%
\end{lem}

\begin{rem}
Note that the accumulation points of the sequence $A_{n}\cap B_{n}$
could be much smaller than $\left(\lim\limits _{n\to\infty}A_{n}\right)\cap\left(\lim\limits _{n\to\infty}B_{n}\right)$.
For example $A_{n}=\left\{ \frac{a}{n}:\,a=0,1,...,n\right\} $ and
$B_{n}=\left(\alpha+A_{n}\text{ (mod 1)}\right)\cup\left\{ 0\right\} $,
for any irrational $\alpha\in\left[0,1\right]$. In this case $A_{n}\cap B_{n}=\left\{ 0\right\} $
for every $n$, but $\lim\limits _{n\to\infty}A_{n}=\lim\limits _{n\to\infty}B_{n}=\left[0,1\right]$.
\end{rem}

In the case where instead of the sequence $B_{n}$ we fix a constant
set $B$, changing the order of the intersection and taking the limit
can only cause a difference on the boundary of $B$.
\begin{lem}
\label{lem:convergence of intersection of a sequence and a compact set}Let
$A_{n}\subseteq\mathbb{R}^{d}$ be a sequence of nonempty compact
sets. Let $A,\,B\subset\mathbb{R}^{d}$ be nonempty compact sets,
and assume that $A_{n}\underset{n\to\infty}{\longrightarrow}A$. Suppose
that $W$ is an accumulation point of $A_{n}\cap B$, then 
\[
A\cap B^{\circ}\subseteq W\subseteq A\cap B.
\]
\end{lem}

Note that the sequence $A_{n}\cap B$ need not converge, but the containment
of the sequence in a compact space guarantees the existence of accumulation
points, whenever $A_{n}\cap B\neq\emptyset$ for infinitely many values
of $n$.%

The difference on the boundary of the set $B$ may be eliminated by
approximating $B$ from the outside in the following way:
\begin{lem}
\label{lem: convergence of a sequence and thickenings of a set}Let
$A_{n,}\,B_{n}$ be two sequences of nonempty compact sets in $\mathbb{R}^{d}$,
such that $A_{n}\underset{n\to\infty}{\longrightarrow}A,\,\text{and }B_{n}\underset{n\to\infty}{\longrightarrow}B$.
Denote $\delta_{n}=d_{H}\left(A_{n},A\right)$, and assume that $B^{\left(\delta_{n}\right)}\subseteq B_{n}$
for every $n$. Then 
\[
A_{n}\cap B_{n}\underset{n\to\infty}{\longrightarrow}A\cap B.
\]
\end{lem}

\begin{lem}
\label{lem:uniform convergence of functions preserve Hausdorff convergence}Let
$A_{n}$ be a sequence of nonempty compact sets in $\mathbb{R}^{d}$,
and assume that $A_{n}\underset{n\to\infty}{\longrightarrow}A\subseteq\rd$.
Let $C\subseteq\mathbb{R}^{d}$ be a compact neighborhood of $A$
(i.e., C is compact and $A\subseteq C^{\circ}$). Let $\left(f_{n}\right)_{n\in\mathbb{N}}$
be a sequence of maps of $\mathbb{R}^{d}$, and assume it converges
uniformly on $C$ to a continuous function $f$. Then $f_{n}\left(A_{n}\right)\converge f\left(A\right)$.

\end{lem}

\subsection{Iterated function systems\label{subsec:Iterated-function-systems}}

\subsubsection{Basic definitions}

For a great introduction to fractal geometry, which covers iterated
function systems and their attractors, the unfamiliar readers are
referred to \cite{falconer2013fractal}.

An iterated function system (IFS) is a finite collection $\Phi=\left\{ \varphi_{i}\right\} _{i\in\Lambda}$
of contracting Lipschitz maps of $\mathbb{R}^{d}$, i.e., for each
$\varphi\in\Phi$ there exists some $r\in\left(0,1\right)$ such that
for every $x,y\in\mathbb{R}^{d}$, 
\[
\left\Vert \varphi\left(x\right)-\varphi\left(y\right)\right\Vert \leq r\cdot\left\Vert x-y\right\Vert .
\]
A basic result due to Hutchinson \cite{Hutchinson1981} states that
every IFS $\Phi=\left\{ \varphi_{i}\right\} _{i\in\Lambda}$, defines
a unique compact set $K\subseteq\mathbb{R}^{d}$, which is called
the \emph{attractor} of the IFS\emph{, }that satisfies\emph{
\[
K=\bigcup\limits _{i\in\Lambda}\varphi_{i}\left(K\right).
\]
}

Whenever the IFS is consisted of similarity maps, the attractor is
called \emph{self-similar }and $\Phi$ is referred to as a \emph{similarity
IFS.} This means that each $\varphi_{i}\in\Phi$ has the form $\varphi_{i}\left(x\right)=r_{i}\cdot O_{i}x+\alpha_{i}$
where $r_{i}\in\left(0,1\right)$ is called the \emph{contraction
ratio} of $\varphi_{i}$, $O_{i}$ is an orthogonal map, and $\alpha_{i}\in\mathbb{R}^{d}$.
A special case is when these similarity maps have trivial orthogonal
parts. Such maps are called \emph{homotheties}, and an attractor of
an IFS which is consisted of homotheties is called \emph{self-homothetic}.

In this paper we only deal with similarity IFSs, and so all IFSs in
this paper are assumed to be similarity IFSs, even when not explicitly
stated.

An important feature of attractors of IFSs is that their points may
be coded in a space of symbols. Suppose that $K$ is the attractor
of the IFS $\Phi=\left\{ \varphi_{i}\right\} _{i\in\Lambda}$, the
map $\gamma_{\Phi}:\Lambda^{\mathbb{N}}\to K$ which is defined by 

\[
\forall i=\left(i_{1},i_{2},...\right)\in\Lambda^{\mathbb{N}},\,\,\gamma_{\Phi}\left(i\right)=\lim_{n\rightarrow\infty}\varphi_{i_{1}}\circ\varphi_{i_{2}}\circ\cdots\circ\varphi_{i_{n}}(0),
\]

is called the \emph{coding map }of $\Phi$. The point $0$ was chosen
arbitrarily and any other base point chosen instead would yield exactly
the same function. The map $\gamma_{\Phi}$ is surjective (but in
general not 1-to-1) and continuous with respect to the product topology
on $\Lambda^{\mathbb{N}}$. 

Given a similarity IFS $\Phi=\left\{ \varphi_{i}\right\} _{i\in\Lambda}$,
for every finite sequence $i=i_{1}...i_{k}\in\Lambda^{*}$ ,we denote
$\varphi_{i}=\varphi_{i_{1}}\circ\cdots\circ\varphi_{i_{k}}$. We
also denote by $r_{i},\,O_{i},\alpha_{i}$ the scaling ratio, orthogonal
part, and translation of the similarity function $\varphi_{i}$ so
that $\varphi_{i}\left(x\right)=r_{i}\cdot O_{i}\left(x\right)+\alpha_{i}$,
in particular, $r_{i}=r_{i_{1}}\cdots r_{i_{k}}$. Sets of the form
$\varphi_{i}K$ for any $i\in\Lambda^{*}$ are called cylinder sets.
Note that for every $i\in\Lambda^{*}$, $\varphi_{i}K=\gamma_{\Phi}\left(\left[i\right]\right)$.

Since the contraction ratios of the functions in the IFS may be different,
it will often be useful to consider the sections $\Pi_{\rho}$ defined
by
\[
\Pi_{\rho}=\left\{ i\in\mathbb{A}^{*}:\,r_{i}\leq\rho<r_{i_{1}}\cdots r_{i_{\left|i\right|-1}}\right\} ,
\]
for any $\rho\in\left(0,r_{\min}\right)$, where $r_{\min}=\min\left\{ r_{i}:\,i\in\Lambda\right\} $.
Note that for every $i\in\Pi_{\rho}$,
\begin{equation}
\rho\cdot r_{\min}<r_{i}\leq\rho.\label{eq:bounds on contraction in sections}
\end{equation}

\subsubsection{Coding}

Given a tree $T\in\mathscr{T}_{\Lambda}$, one may apply the coding
map to the boundary of $T$ resulting in some compact subset of the
attractor $K$. One issue here is that if $T$ is a finite tree it
has an empty boundary and so an application of the coding map would
yield the empty set which is not a member of $\Omega_{K}$. Considering
only infinite trees, the coding map $\gamma_{\Phi}:\Lambda^{\mathbb{N}}\to K$
induces a function $\Gamma_{\Phi}:\mathscr{T}_{\Lambda}^{\infty}\to\Omega_{K}$
given by:

\[
\Gamma_{\Phi}\left(T\right)=\gamma_{\Phi}\left(\partial T\right).
\]

Note that $\Gamma_{\Phi}$ is a surjection%
, mapping the full tree $\Lambda^{*}$ to $K$. However, one should
note that for a tree $T\in\mathscr{T}_{\Lambda}^{\infty}$, nodes
of $T$ that have a finite line of descendants have no influence on
the image $\Gamma_{\Phi}\left(T\right)$. Therefore, it sometimes
makes sense to consider the restriction of $\Gamma_{\Phi}$ to the
subspace $\mathbb{\mathscr{T}}_{\Lambda}^{\prime}$. The map $\Gamma_{\Phi}$
restricted to the subspace $\mathbb{\mathscr{T}}_{\Lambda}^{\prime}$
has a much nicer behavior, in particular, as shown in the following
lemma, it is continuous. 
\begin{lem}
\label{lem:Projection of trees without leaves is continuous}The function
$\Gamma_{\Phi}\restriction_{\mathbb{\mathscr{T}}_{\Lambda}^{\prime}}:\mathbb{\mathscr{T}}_{\Lambda}^{\prime}\to\Omega_{\Phi}$
is continuous.
\end{lem}

\begin{proof}
Let $\left(T_{\left(n\right)}\right)_{n\in\mathbb{N}}$ be a converging
sequence of trees in $\mathbb{\mathscr{T}}_{\Lambda}^{\prime}$, and
denote $\lim T_{\left(n\right)}=T$. We need to show that $\Gamma_{\Phi}\left(T_{\left(n\right)}\right)\underset{n\to\infty}{\longrightarrow}\Gamma_{\Phi}\left(T\right)$.
Fix any $r_{\Phi}\in\left(0,1\right)$ which satisfies 
\[
\left\Vert \varphi_{i}\left(x\right)-\varphi_{i}\left(y\right)\right\Vert <r_{\Phi}\cdot\left\Vert x-y\right\Vert ,
\]
 for every $i\in\Lambda$ and $x,y\in\mathbb{R}^{d}$. Given $\varepsilon>0$,
fix some $l_{\varepsilon}\in\N$, such that $\left(r_{\Phi}\right)^{l_{\varepsilon}}<\varepsilon/\text{diam}\left(K\right)$.
This implies that for every $i\in\Lambda^{l_{\varepsilon}}$, the
diameter of the cylinder set $\gamma_{\Phi}\left[i\right]$ is smaller
than $\varepsilon$.

Since $T_{\left(n\right)}\converge T$, there is some $N_{\varepsilon}\in\mathbb{N}$,
such that whenever $n>N_{\varepsilon}$, $\left(T_{\left(n\right)}\right)_{t}=T_{t}$
for every $t=0,1,...,l_{\varepsilon}$. Assume now that $n>N_{\varepsilon}.$ 

For every $x\in\Gamma_{\Phi}\left(T\right)$, there is some $i\in\partial T$,
such that $x=\gamma_{\Phi}\left(i\right).$ Since $n>N_{\varepsilon}$,
there is some $j\in\partial T_{\left(n\right)}$ such that $j\in\left[i_{1}...i_{l_{\varepsilon}}\right]$.
Denote $y=\gamma_{\Phi}\left(j\right)\in\Gamma_{\Phi}\left(T_{\left(n\right)}\right)$.
Then $x,y\in\gamma_{\Phi}\left[i_{1}...i_{l_{\varepsilon}}\right]$,
and since $\gamma_{\Phi}\left[i_{1}...i_{l_{\varepsilon}}\right]$
has diameter smaller than $\varepsilon$, $\left\Vert x-y\right\Vert <\varepsilon$. 

On the other hand, exactly in the same way as above, for every $y\in\Gamma_{\Phi}\left(T_{\left(n\right)}\right)$,
there is some $j\in\partial T_{\left(n\right)}$, such that $y=\gamma_{\Phi}\left(j\right)$.
Since $n>N_{\varepsilon}$, there is some $i\in\partial T$ such that
$i\in\left[j_{1}...j_{l_{\varepsilon}}\right]$, and denoting $x=\gamma_{\Phi}\left(i\right)$,
we have found an $x\in\Gamma_{\Phi}\left(T\right)$, such that $\left\Vert x-y\right\Vert <\varepsilon$.
\end{proof}
\begin{rem}
Note that $\Gamma_{\Phi}\restriction_{\mathbb{\mathscr{T}}_{\Lambda}^{\prime}}$
still need not be 1-to-1, which follows from the fact that $\gamma_{\Phi}$
is not 1-to-1 in general.
\end{rem}

Now, since $\partial T=\partial T^{\prime}$ for every tree $T\in\mathscr{T}_{\Lambda}^{\infty}$,
we may decompose $\Gamma_{\Phi}:\mathscr{T}_{\Lambda}^{\infty}\to\Omega_{K}$
into $T\mapsto T^{\prime}\mapsto\Gamma_{\Phi}\left(T^{\prime}\right)$.
We have already shown that the second part, $T^{\prime}\mapsto\Gamma_{\Phi}\left(T^{\prime}\right)$,
is continuous. However, the map $T\mapsto T^{\prime}$ is not continuous.
For example, one can pick any infinite sequence $i\in\Lambda^{\mathbb{N}}$,
with the corresponding ray $S=\left\{ \emptyset\right\} \cup\bigcup\limits _{k=1}^{\infty}\left\{ i_{1}\dots i_{k}\right\} $,
and consider the sequence of infinite trees $\left(T_{\left(n\right)}\right)_{n\in\mathbb{N}}$
defined by $T_{\left(n\right)}=S\cup\bigcup\limits _{k=0}^{n}\Lambda^{k}$.
Obviously, $T_{\left(n\right)}\converge\Lambda^{*}$, while for every
$n$, $\left(T_{\left(n\right)}\right)^{\prime}=S$. 

On the bright side, the map $T\mapsto T^{\prime}$ is Borel (which
is the content of the following proposition), and therefore, so is
the map $\Gamma_{\Phi}$.
\begin{prop}
The map $R:\mathscr{T}_{\Lambda}^{\infty}\to\mathbb{\mathscr{T}}_{\Lambda}^{\prime}$
, defined by $R\left(T\right)=T^{\prime}$ is Borel.
\end{prop}

\begin{proof}
A basic open set in $\mathbb{\mathscr{T}}_{\Lambda}^{\prime}$ is
of the form $\left[F\right]\cap\mathbb{\mathscr{T}}_{\Lambda}^{\prime}$
for a finite tree $F$. Note that in order for the intersection to
not be empty, $F$ must be a prefix of some tree in $\mathbb{\mathscr{T}}_{\Lambda}^{\prime}$,
i.e., each vertex $v\in F$ with $\left|v\right|<\text{height}\left(F\right)$
must have at least one child. It is enough%
{} to show that for such a set, $R^{-1}\left(\left[F\right]\right)$
is Borel %
.

This may be proved by induction on the height of $F$. For $F=\emptyset$,
$R^{-1}\left(\left[F\right]\right)=\mathscr{T}_{\Lambda}^{\infty}$.
For the induction step we need the following claim:
\begin{claim*}
Given any $v\in\Lambda^{*}$, the set $\mathscr{V}_{v}=\left\{ T\in\mathscr{T}_{\Lambda}:\,v\in T\,\text{and}\,T^{v}\in\mathscr{T}_{\Lambda}^{\infty}\right\} $
is Borel, and the same thing holds for the set $\mathscr{W}_{v}=\left\{ T\in\mathscr{T}_{\Lambda}:\,v\in T\,\text{and}\,T^{v}\notin\mathscr{T}_{\Lambda}^{\infty}\right\} $.
\end{claim*}
\begin{proof}[Proof of Claim]
Denote
\[
\mathscr{S}_{v}=\left\{ L\in\mathscr{T}_{\Lambda}\setminus\mathscr{T}_{\Lambda}^{\infty}:\,v\in L,\,\text{height}\left(L^{v}\right)+\left|v\right|<\text{height}\left(L\right)\right\} .
\]
This is a countable collection since the set of all finite trees itself
is countable. Note that 
\[
\mathscr{W}_{v}=\left(\bigcup_{L\in\mathscr{S}_{v}}\left[L\right]\right)\cup\left\{ T\in\mathscr{T}_{\Lambda}\setminus\mathscr{T}_{\Lambda}^{\infty}:\,v\in T\right\} .
\]
Since $\bigcup\limits _{L\in\mathscr{S}_{v}}\left[L\right]$ is an
open set and $\left\{ T\in\mathscr{T}_{\Lambda}\setminus\mathscr{T}_{\Lambda}^{\infty}:\,v\in T\right\} $
is countable (and hence Borel), $\mathscr{W}_{v}$ is indeed a Borel
set. Moreover, since the set $\left\{ T\in\mathscr{T}_{\Lambda}:\,v\in T\right\} $
is Borel (in fact, it is open)%
, and 
\[
\mathscr{V}_{v}=\left\{ T\in\mathscr{T}_{\Lambda}:\,v\in T\right\} \backslash\mathscr{W}_{v},
\]
$\mathscr{V}_{v}$ is also a Borel set.
\end{proof}
Now, for the induction step, we assume that $R^{-1}\left(\left[F\right]\right)$
is Borel for every $F$ of height $n-1$. Given a finite tree $F$
of height $n$, denote $F^{-}=\bigcup\limits _{k=0}^{n-1}F_{k}$,
then for a tree $T\in\mathscr{T}_{\Lambda}^{\infty}$, $T^{\prime}\in\left[F\right]$
iff $T^{\prime}\in\left[F^{-}\right]$, $F_{n}\subseteq T$, and for
every $v\in T_{n}$, $T^{v}\in\mathscr{T}_{\Lambda}^{\infty}$ iff
$v\in F_{n}$. To formulate this, 
\[
\begin{array}{lcl}
R^{-1}\left(\left[F\right]\right) & = & R^{-1}\left(\left[F^{-}\right]\right)\cap\\
 &  & \left(\bigcap_{v\in F_{n}}\left\{ T\in\mathscr{T}_{\Lambda}:\,v\in T\,\text{and}\,T^{v}\in\mathscr{T}_{\Lambda}^{\infty}\right\} \right)\cap\\
 &  & \left(\bigcap_{v\in\Lambda^{n}\backslash F_{n}}\left(\left\{ T\in\mathscr{T}_{\Lambda}:\,v\notin T\right\} \cup\left\{ T\in\mathscr{T}_{\Lambda}:\,v\in T\,\text{and}\,T^{v}\notin\mathscr{T}_{\Lambda}^{\infty}\right\} \right)\right)\\
 & = & R^{-1}\left(\left[F^{-}\right]\right)\cap\left(\bigcap\limits _{v\in F_{n}}\mathscr{V}_{v}\right)\cap\left(\bigcap\limits _{v\in\Lambda^{n}\backslash F_{n}}\left(\left\{ T\in\mathscr{T}_{\Lambda}:\,v\notin T\right\} \cup\mathscr{W}_{v}\right)\right)
\end{array}
\]
Now, using the assumption that $R^{-1}\left(\left[F^{-}\right]\right)$
is Borel and the claim above, it may be easily seen that $R^{-1}\left(\left[F\right]\right)$
is Borel.
\end{proof}
\begin{prop}
\label{prop:Gamma_=00005CPhi is Borel}The map $\Gamma_{\Phi}:\mathscr{T}_{\Lambda}^{\infty}\to\Omega_{K}$
is Borel.
\end{prop}

\subsubsection{Separation conditions\label{subsec:Separation-conditions}}

Recall that given an IFS $\Phi=\left\{ \varphi_{i}\right\} _{i\in\Lambda}$,
the attractor of the IFS is the unique compact set $K$ which satisfies
$K=\bigcup\limits _{i\in\Lambda}\varphi_{i}K$. Of course, in general,
the sets $\varphi_{i}K$ may intersect each other, and so there are
several ``separation properties'' dealing with the size of these
intersections. Here are the two most commonly used separation conditions:
\begin{defn}
\label{def:OSC and SSC}Given an IFS $\Phi=\left\{ \varphi_{i}\right\} _{i\in\Lambda}$
whose attractor is denoted by $K$, we define the following properties:
\begin{enumerate}
\item $\Phi$ satisfies the \emph{strong separation condition (SSC)} if
for every $i\neq j\in\Lambda,\,\varphi_{i}\left(K\right)\cap\varphi_{j}\left(K\right)=\emptyset$.
\item $\Phi$ satisfies the \emph{open set condition (OSC)} if there exists
a non-empty open set $U\subseteq\mathbb{R}^{d}$, such that $\varphi_{i}U\subseteq U$
for every $i\in\Lambda$, and for every $i\neq j\in\Lambda,\,\varphi_{i}U\cap\varphi_{j}U=\emptyset$.
In that case we call $U$ an \emph{OSC set.}
\end{enumerate}
\end{defn}

Note that whenever $U$ is an OSC set, $K\subseteq\overline{U}$.
Roughly speaking, the OSC says that two sets $\varphi_{i}\left(K\right),\varphi_{j}\left(K\right)$
for $i\neq j\in\Lambda$ may only intersect on their boundaries. The
SSC is clearly stronger than the OSC. In fact, the following holds:
\begin{lem}
\label{lem:SSC implies OSC set containing the attractor}A similarity\textup{\emph{
IFS}} $\Phi=\left\{ \varphi_{i}\right\} _{i\in\Lambda}$ with an attractor
$K$, satisfies the SSC iff it has an OSC set $U$ such that $K\subseteq U$.%
\end{lem}

A stronger condition than the OSC is the so called \emph{strong OSC}
which assumes the existence of an OSC set $U$ such that $U\cap K\neq\emptyset$
. It was shown by Schief \cite{schief1994separation} that for similarity\footnote{The same was shown later for conformal IFSs in \cite{Peres2001equivalence},
and for semi-conformal IFSs in \cite{KaenmakiVilppolainen2008}} IFSs the OSC and the strong OSC are equivalent.

One obvious obstruction for the OSC to hold is the existence of overlaps
in the IFS, i.e., the existence of $i\neq j\in\Lambda^{*}$ such that
$\varphi_{i}=\varphi_{j}$. Aiming to generalize the OSC to allow
overlaps, Lau and Ngai defined a weaker condition called the \emph{weak
separation condition (WSC). }The definition given below is in general
slightly stronger than the original definition given in \cite{LauNgai1999}.
However, it was shown in \cite{Zerner1996weak} that when the attractor
of the IFS is not contained in an affine hyperplane these two definitions
are equivalent. 
\begin{defn}
We say that an IFS $\Phi=\left\{ \varphi_{i}\right\} _{i\in\Lambda}$
satisfies the \emph{weak separation condition (WSC)} if there exists
some $c\in\mathbb{N}$, such that for every $\rho>0$, every ball
$B$ of radius $\rho$ intersects at most $c$ different cylinder
sets $\varphi_{i}\left(K\right)$ for $i\in\Pi_{\rho}$, i.e., 
\[
\left|\left\{ \varphi_{i}\left(K\right):\,\varphi_{i}\left(K\right)\cap B\neq\emptyset,\,i\in\Pi_{\rho}\right\} \right|\leq c,
\]
 where identical cylinder sets $\varphi_{i}\left(K\right)=\varphi_{j}\left(K\right)$
for $i\neq j$ are only counted once.
\end{defn}

It is known (see \cite{Zerner1996weak}) that similarity IFSs satisfy
the OSC if and only if they satisfy the WSC (as defined here) and
they have no overlaps. 
\begin{lem}
\label{lem:SSC and cube intersecting cylinder sets}Let $\Phi=\left\{ \varphi_{i}\right\} _{i\in\Lambda}$
be a similarity IFS in $\mathbb{R}^{d}$ satisfying SSC. Then there
is some constant $\tau>0$ s.t. for every cube $C$ of side-length
$\rho$ small enough, C intersects at most one cylinder set $\varphi_{i}\left(K\right)$
such that $i\in\Pi_{\tau\rho}$.
\end{lem}

\begin{proof}
Denote $\delta=\min\limits _{i\neq j\in\Lambda}\text{dist}\left(\varphi_{i}\left(K\right),\varphi_{j}\left(K\right)\right)$,
where $K$ is the attractor of $\Phi$. It is not hard to see that
the cylinder sets on the section $\Pi_{\frac{\rho\cdot\sqrt{d}}{\delta\cdot r_{\min}}}$
are of distance larger than $\rho\cdot\sqrt{d}$ from each other,
and since $\rho\cdot\sqrt{d}$ is exactly the diameter of the cube
$C$, it cannot intersect mote than one cylinder set $\varphi_{i}\left(K\right)$
for $i\in\Pi_{\frac{\rho\cdot\sqrt{d}}{\delta\cdot r_{\min}}}$. Hence,
the claim is satisfied for $\tau=\frac{\sqrt{d}}{\delta\cdot r_{\min}}$.
\end{proof}
As mentioned above the SSC is stronger than the OSC (follows e.g.
from Lemma \ref{lem:SSC implies OSC set containing the attractor}),
and it is not hard to show that these two properties are not equivalent.
The relations between the separation conditions that were presented
in this subsection are summarized in the following lemma:
\begin{lem}
\label{lem:Separation conditions hierarchy}For similarity IFSs, the
separation conditions discussed above satisfy the following hierarchy:
\[
\text{SSC }\begin{array}{c}
\nLeftarrow\\
\Rightarrow
\end{array}\begin{array}{c}
\text{OSC}\\
\Updownarrow\\
\text{strong OSC}
\end{array}\begin{array}{c}
\nLeftarrow\\
\Rightarrow
\end{array}\text{ WSC}
\]
\end{lem}

\subsection{Assouad and lower dimensions\label{subsec:Assouad and lower dimensions}}

As the Assouad dimension and the lower dimension are central objects
of study in this paper, we now provide their formal definitions.
We use the notation $N_{r}\left(F\right)$ to denote the smallest
number of (open) balls of radius $r$ needed to cover a set $F\subseteq\rd$.
\begin{defn}
Given a non-empty set $S\subseteq\mathbb{R}^{d}$, the Assouad dimension
of $S$ is defined by: 
\[
\begin{array}{ccc}
\dima\left(S\right)= & \inf & \biggl\{\alpha\geq0:\,\exists C>0,\,\forall0<r<R,\,\forall x\in S,\\
 &  & \left.N_{r}\left(B_{R}\left(x\right)\cap S\right)\leq C\cdot\left(\dfrac{R}{r}\right)^{\alpha}\right\} 
\end{array}
\]
\end{defn}

The Lower dimension is a dual notion of dimension defined as follows:
\begin{defn}
Given a non-empty set $S\subseteq\mathbb{R}^{d}$,
\[
\begin{array}{ccc}
\diml\left(S\right)= & \sup & \biggl\{\alpha\geq0:\,\exists C>0,\,\forall0<r<R\leq\text{diam}\left(S\right),\,\forall x\in S,\\
 &  & \left.N_{r}\left(B_{R}\left(x\right)\cap S\right)\geq C\cdot\left(\dfrac{R}{r}\right)^{\alpha}\right\} 
\end{array}
\]
\end{defn}

As mentioned in the introduction, Theorem \ref{thm:Assouad as max}
may be considered as an alternative (but equivalent) definition of
$\diml$ and $\dima$ for compact sets. In general, for a compact
set $F\subseteq\rd$, 
\[
\diml\left(F\right)\leq\dimh\left(F\right)\leq\dima\left(F\right).
\]
For a thorough exposition to these notions, the reader is referred
to \cite{Fraser2020}.

\section{Microsets via coding \label{sec:Microsets-via-coding}}

\subsection{Limits of descendant trees and their corresponding sets\label{subsec:Limits-of-descendant-trees}}

Given any tree $T$, we denote the set of all the descendant trees
of vertices of $T$ by $\mathscr{D}_{T}=$$\left\{ T^{v}:\,v\in T\right\} $.
Let $\Phi=\left\{ \varphi_{i}\right\} _{i\in\Lambda}$ be an IFS whose
attractor is denoted by $K$, and $T\in\mathscr{T}_{\Lambda}^{\prime}$.
Every descendant tree of $T$ has no leaves and may be mapped using
$\Gamma_{\Phi}$ to $\Omega_{K}$. We denote $\mathcal{S}_{\Phi}^{T}=\Gamma_{\Phi}\left(\mathscr{D}_{T}\right)$.
Note that since $\Gamma_{\Phi}\restriction_{\mathscr{T}_{\Lambda}^{\prime}}$
is continuous (Lemma \ref{lem:Projection of trees without leaves is continuous})
and closed (by compactness),
\begin{equation}
\overline{\mathcal{S}_{\Phi}^{T}}=\Gamma_{\Phi}\left(\overline{\mathscr{D}_{T}}\right),\label{eq:change order of closure and Gamma_=00005CPhi}
\end{equation}
where as usual, the topological closure of a set is denoted by an
overline, and so $\overline{\mathscr{D}_{T}}$ is the closure of $\mathscr{D}_{T}$
in the Hausdorff metric, and $\overline{\mathcal{S}_{\Phi}^{T}}$
is the closure of $\mathcal{S}_{\Phi}^{T}$ in the space of trees
with the topology described in subsection \ref{subsec:Symbols-and-trees}.
A simple, but important, observation is that if $T=\Lambda^{*}$,
i.e., $T$ is the full tree which codes the entire self-similar set
$K$, then $\overline{\mathcal{S}_{\Phi}^{T}}=\left\{ K\right\} $.
\begin{lem}
\label{lem:branch set of a branch set}Given an alphabet $\mathbb{A}$
and a tree $T\in\mathscr{T}_{\mathbb{A}}$, for every $S\in\overline{\mathcal{\mathcal{D}}_{T}}$,~$\overline{\mathcal{\mathcal{D}}_{S}}\subseteq\overline{\mathcal{\mathcal{D}}_{T}}$.
\end{lem}

\begin{proof}
First, note that without the closures, the claim is obvious, since
for $v\in T$ and $w\in T^{v}$, $\left(T^{v}\right)^{w}=T^{vw}$.

Now, for $S\in\overline{\mathcal{\mathcal{D}}_{T}}$, let $S_{\left(n\right)}\in\mathcal{\mathcal{D}}_{T}$
be a sequence of trees converging to $S$, and let $v\in S$ be any
node of $S$. We show that $S^{v}\in\overline{\mathcal{\mathcal{D}}_{T}}$.
Indeed, since $S_{\left(n\right)}\converge S$, whenever $n$ is large
enough $v\in S_{\left(n\right)}$, and it is easy to see that $S_{\left(n\right)}^{v}\converge S^{v}$.
Since for every $n$, $S_{\left(n\right)}\in\mathcal{\mathcal{D}}_{T}$,
we have $S_{\left(n\right)}^{v}\in\mathcal{\mathcal{D}}_{T}$ as well,
and therefore $S^{v}\in\overline{\mathcal{\mathcal{D}}_{T}}$.
\end{proof}
The goal in this section of the paper is to investigate the relation
between the collection of microsets of a given compact set, and the
collection $\overline{\mathcal{S}_{\Phi}^{T}}$, for a suitable coding
tree $T$ for that set. Of particular interest is the set of Hausdorff
dimensions of sets in $\overline{\mathcal{S}_{\Phi}^{T}}$ (and their
minisets) and its relation with Hausdorff dimensions of microsets.

\subsection{Microsets as limits of descendant trees}
\begin{proof}[Proof of Theorem \ref{thm:microset is a union of minisets}]
Given a microset $M$ of $F$, there are minisets of $F$ converging
to $M$. That is, there are similarity functions $\left(f_{i}\right)_{i\in\mathbb{N}}$,
each of the form $f_{i}\left(x\right)=O_{i}\left(\lambda_{i}\cdot x\right)+y_{i},$
where $O_{i}$ is an orthogonal transformation of $\mathbb{R}^{d}$,
$\lambda_{i}\in\left(1,\infty\right),$ $y_{i}\in\mathbb{R}^{d}$,
such that
\[
\lim_{i\to\infty}Q\cap f_{i}\left(F\right)=M.
\]

First we consider the easy case where the sequence $\left(\lambda_{i}\right)_{i\in\N}$
is bounded. In that case, after moving to a subsequence, the similarities
$f_{i}$ converge pointwise to some similarity function $\psi$, where
the convergence is uniform on compact subsets of $\mathbb{R}^{d}$%
. Hence, by Lemma \ref{lem:uniform convergence of functions preserve Hausdorff convergence}
and Lemma \ref{lem:convergence of intersection of a sequence and a compact set},
\[
Q^{\circ}\cap\psi\left(F\right)\subseteq M\subseteq Q\cap\psi\left(F\right).
\]

From this point forward we assume the sequence $\left(\lambda_{i}\right)_{i\in\mathbb{N}}$
is unbounded. Note that 
\[
Q\cap f_{i}\left(F\right)=f_{i}\left(f_{i}^{-1}\left(Q\right)\cap F\right).
\]
$f_{i}^{-1}\left(Q\right)$ is a cube of side-length $\lambda_{i}^{-1}$,
hence, it follows from the WSC that there is some $n\geq1$ such that
$f_{i}^{-1}\left(Q\right)$ intersects at most $n$ distinct cylinder
sets on the section $\Pi_{\lambda_{i}^{-1}/r_{\min}}$,i.e., there
exist vertices $v_{i,1},...,v_{i,n}\in T\cap\Pi_{\lambda_{i}^{-1}/r_{\min}}$
(with possible repetitions if there are less than $n$ distinct vertices)
such that 
\[
f_{i}^{-1}\left(Q\right)\cap F=f_{i}^{-1}\left(Q\right)\cap\left(\bigcup_{j=1}^{n}\varphi_{v_{i,j}}\gamma_{\Phi}\left(\partial T^{v_{i,j}}\right)\right).
\]
This implies that 
\[
\begin{array}{l}
Q\cap f_{i}\left(F\right)=\\
f_{i}\left(f_{i}^{-1}\left(Q\right)\cap F\right)=\\
f_{i}\left(f_{i}^{-1}\left(Q\right)\cap\left(\bigcup\limits _{j=1}^{n}\varphi_{v_{i,j}}\gamma_{\Phi}\left(\partial T^{v_{i,j}}\right)\right)\right)=\\
Q\cap\left(\bigcup\limits _{j=1}^{n}f_{i}\circ\varphi_{v_{i,j}}\gamma_{\Phi}\left(\partial T^{v_{i,j}}\right)\right)
\end{array}
\]
For convenience we denote $A_{i,j}=\gamma_{\Phi}\left(\partial T^{v_{i,j}}\right)$
and $\psi_{i,j}=f_{i}\circ\varphi_{v_{i,j}}$, then 
\[
Q\cap f_{i}\left(F\right)=Q\cap\left(\bigcup_{j=1}^{n}\psi_{i,j}\left(A_{i,j}\right)\right)
\]

Since $v_{i,j}\in\Pi_{\lambda_{i}^{-1}/r_{\min}}$, 
\[
\lambda_{i}^{-1}<r_{v_{i,j}}\leq\lambda_{i}^{-1}/r_{\min}.
\]
Therefore, the scaling ratio of $\psi_{i,j}$, which is $r_{v_{i,j}}\cdot\lambda_{i}$
is in the interval $\left(1,\,1/r_{\min}\right]$. It is not hard
to observe that the translation parts of $\varphi_{v_{i,j}}$ and
$f_{i}$ are also bounded, and hence so is the translation part of
$\psi_{i,j}$. This implies that by moving to a subsequence, the similarities
$\psi_{i,j}$ converge as $i\to\infty$ to some similarity function
$\psi_{j}$, and the convergence is uniform on compact sets.

Note that $A_{i,j}\in\mathcal{S}_{\Phi}^{T}$. Since $\overline{\mathcal{S}_{\Phi}^{T}}$
is compact w.r.t. the Hausdorff metric, by passing to a subsequence
we may also assume that for every $j,$ $A_{i,j}$ converge to some
compact set $A_{j}\in\overline{\mathcal{S}_{\Phi}^{T}}$ as $i\to\infty$.
Overall, using Lemma \ref{lem:uniform convergence of functions preserve Hausdorff convergence}
the following convergence holds in the Hausdorff metric: 
\[
\bigcup_{j=1}^{n}\psi_{i,j}\left(A_{i,j}\right)\underset{i\to\infty}{\longrightarrow}\bigcup_{j=1}^{n}\psi_{j}\left(A_{j}\right)
\]
Applying Lemma \ref{lem:convergence of intersection of a sequence and a compact set},
Equation \ref{eq:microsets are minisets} now follows.

In case the SSC holds, since we assume the scaling sequence $\left(\lambda_{i}\right)$
is unbounded, we may use Lemma \ref{lem:SSC and cube intersecting cylinder sets},
assuming $\lambda_{i}^{-1}$ is small enough for all $i$. Then the
proof proceeds exactly as above but with $n=1$.
\end{proof}
\begin{rem}
Since the set $\Lambda^{*}$ is countable, there is no need to use
the axiom of choice during the proof,
\end{rem}

\begin{rem}
Note that the expression $Q\cap\left(\bigcup_{j=1}^{n}\psi_{i}A_{i}\right)$
on the right hand side of Equation (\ref{eq:microsets are minisets}),
is nothing but a finite union of minisets of the sets $A_{1},...,A_{n}\in\overline{\mathcal{S}_{\Phi}^{T}}$,
and the set on the left hand side of Equation (\ref{eq:microsets are minisets})
is the union of those minisets intersected with $Q^{\circ}$. Also
note that Equation (\ref{eq:microsets are minisets}) implies that
\[
Q^{\circ}\cap\left(\bigcup_{j=1}^{n}\psi_{i}A_{i}\right)=Q^{\circ}\cap M.
\]
In particular, since $M\in\mathcal{W}_{F},$$Q^{\circ}\cap\left(\bigcup_{j=1}^{n}\psi_{i}A_{i}\right)\neq\emptyset$.
\end{rem}

\begin{rem}
One can not expect an equality to hold instead of the inclusion at
the RHS of equation (\ref{eq:microsets are minisets}) (the same holds
for the LHS inclusion, but that is trivial). Consider for example
the middle thirds Cantor set $\mathscr{C}\subset\mathbb{R}.$ The
set 
\[
\left[0,1\right]\cap\left(\frac{9}{7}\cdot\left(\mathscr{C}-\frac{1}{9}\right)\right)\setminus\left\{ 0,1\right\} 
\]
is a microset of $\mathscr{C}$ but it is not a miniset.
\end{rem}

\begin{lem}
\label{lem:miniset in Q interior}Let $S\subseteq\mathbb{R}^{d}$
be a compact set, and let $A\subseteq Q$ be a miniset of $S$ s.t.
$A\cap Q^{\circ}\neq\emptyset$. Then there exists a non-empty set
$B\subseteq Q$ with $B\cap Q^{\circ}\neq\emptyset$, s.t. $B$ is
a miniset of $S$ and $\dimh\left(B\right)\leq\dimh\left(A\cap Q^{\circ}\right)$
\end{lem}

\begin{proof}
By assumption, there is some point $x\in A\cap Q^{\circ}$, and hence
there is some smaller closed cube $C$ whose edges are parallel to
the axis such that $x\in C\subseteq Q^{\circ}$. Denote by $F_{C}$
the unique homothety mapping $C$ to $Q$, and denote $B=F_{C}\left(C\cap A\right)$.
Clearly, $B$ is a miniset of $S$, and $\dimh\left(B\right)\leq\dimh\left(A\cap Q^{\circ}\right)$.

Combining Theorem \ref{thm:microset is a union of minisets} and Lemma
\ref{lem:miniset in Q interior}, the following corollary is immediate:
\end{proof}
\begin{cor}
Under the assumptions of Theorem \ref{thm:microset is a union of minisets},
for every microset $M\in\mathcal{W}_{F}$, there is a set $S\in\overline{\mathcal{S}_{\Phi}^{T}}$,
and two minisets of $S$, $A$ and $B$ which intersect $Q^{\circ}$,
such that
\[
\dimh\left(A\right)\leq\dimh\left(M\right)\leq\dimh\left(B\right).
\]

\end{cor}

\subsection{Hausdorff dimensions of sets in $\overline{\mathcal{S}_{\Phi}^{T}}$
vs. Hausdorff dimensions of their minisets.\label{subsec:Hausdorff dim of branch sets vs their minisets}}
\begin{prop}
\label{prop:minisets of branch sets in int(Q) vs branch sets}Let
$\Phi=\left\{ \varphi_{i}\right\} _{i\in\Lambda}$ be a similarity
IFS and $T\in\mathscr{T}_{\Lambda}^{\prime}$. The following holds:

\[
\overline{\left\{ \dimh\left(B\cap Q^{\circ}\right):B\text{ is a miniset of some}\,S\in\overline{\mathcal{S}_{\Phi}^{T}}\right\} }=\overline{\left\{ \dimh\left(S\right):\,S\in\overline{\mathcal{S}_{\Phi}^{T}}\right\} }.
\]
\end{prop}

\begin{proof}
The inclusion $\supseteq$ is trivial even without taking the closure
of the sets. For the other inclusion, given $B=Q\cap\psi\left(S\right)=\psi\left(\psi^{-1}Q\cap S\right)$,
let $Y\in\overline{\mathscr{D}_{T}}$ be a coding tree such that $\Gamma_{\Phi}\left(Y\right)=S$.
For every $x\in\psi^{-1}Q^{\circ}\cap S$ there is some $v\in Y$
such that $x\in\varphi_{v}\left(\Gamma_{\Phi}\left(Y^{v}\right)\right)\subseteq\psi^{-1}Q^{\circ}\cap S$.
Since as a set $Y$ is countable, there is a sequence $v_{n}\in Y$
such that $\psi^{-1}Q^{\circ}\cap S=\bigcup\limits _{n=1}^{\infty}\varphi_{v_{n}}\left(\Gamma_{\Phi}\left(Y^{v_{n}}\right)\right)$
and therefore, $\dimh\left(B\cap Q^{\circ}\right)=\sup\left\{ \dimh\left(\Gamma_{\Phi}\left(Y^{v_{n}}\right)\right):\,n\in\mathbb{N}\right\} $.
By Lemma \ref{lem:branch set of a branch set}, $\Gamma_{\Phi}\left(Y^{v_{n}}\right)\in\overline{\mathcal{S}_{\Phi}^{T}}$
for every $n$, which implies the inclusion $\subseteq$.
\end{proof}
In the statement of the proposition it is necessary to take the intersections
with the interior of $Q$ rather than $Q$ itself. Consider the following
example in $\mathbb{R}^{2}$:
\begin{example}
Take $F=\left(\left[0,1\right]^{2}\setminus\left[0,\frac{1}{2}\right)^{2}\right)\cup\left\{ \left(\frac{1}{4},\frac{1}{4}\right)\right\} $,
and a coding tree $T$ w.r.t. 2-adic coding, i.e., using the IFS $\Phi=\left\{ \varphi_{i}:\,\left(x,y\right)\mapsto\dfrac{1}{2}\left(x,y\right)+\dfrac{1}{2}i\right\} _{i\in\Lambda}$
where $\Lambda=\left\{ 0,1\right\} ^{2}$. We define $T$ by $T_{0}=\Lambda,\,T^{\left(0,1\right)}=T^{\left(1,0\right)}=T^{\left(1,1\right)}=\Lambda^{*}$,
and $T^{\left(0,0\right)}=\left\{ \emptyset\right\} \cup\bigcup\limits _{n=1}^{\infty}\left\{ i_{1}\cdots i_{n}\right\} $
where $\gamma_{\Phi}\left(i\right)=\left(\frac{1}{4},\frac{1}{4}\right)$,
i.e., $T^{\left(0,0\right)}$ is just a ray corresponding to the point
$\left(\frac{1}{4},\frac{1}{4}\right)$. Then $\dimh\left(\varphi_{\left(0,0\right)}^{-1}\left(F\right)\cap Q\right)=1$
where $\left\{ \dimh\left(S\right):\,S\in\overline{\mathcal{S}_{\Phi}^{T}}\right\} =\left\{ 2,0\right\} $. 
\end{example}

Also note that taking the closure of the sets in Proposition \ref{prop:minisets of branch sets in int(Q) vs branch sets}
is necessary as well, and without it the inclusion $\subseteq$ might
fail. Indeed, consider the following example:
\begin{example}
For the construction of this example we use the IFS $\Phi=\left\{ \varphi_{0}:x\mapsto\frac{1}{2}x,\,\varphi_{1}:x\mapsto\frac{1}{2}x+\frac{1}{2}\right\} $
which corresponds to base 2 digital expansion in $\mathbb{R}$. We
construct the set $F$ around the point $\dfrac{1}{3}$ whose base
2 digital expansion sequence is $010101\dots$. Let $V=\left\{ 1,\,011,\,01011,\,0101011,\dots\right\} $,
so that $V$ is the set of all $v\in\left\{ 0,1\right\} ^{*}$ s.t.
$\forall x\in\varphi_{v}\left(\left[0,1\right]\right)$, $x>\dfrac{1}{3}$
and $\forall t<v$, $\dfrac{1}{3}\in\varphi_{t}\left(\left[0,1\right]\right)$.
We order the set $V$ by the length of its elements, $V=\left\{ v_{1},v_{2},...\right\} $.
We also define $W=\left\{ 00,\,0100,\,010100,\,01010100,\dots\right\} $,
i.e. $W$ is the set of all $w\in\left\{ 0,1\right\} ^{*}$ s.t. $\forall x\in\varphi_{w}\left(\left[0,1\right]\right)$,
$x<\dfrac{1}{3}$ and $\forall t<w$, $\dfrac{1}{3}\in\varphi_{t}\left(\left[0,1\right]\right)$.
Except for nodes along the ray $01010101\dots$, for every $t\in\left\{ 0,1\right\} ^{*}$
there exists a unique $y\in W\cup V$ s.t. $t<y$.

Now, we define a tree $T\in\mathscr{T}_{\left\{ 0,1\right\} }$ as
follows:
\begin{itemize}
\item For every $w\in W$, $T^{w}=\left\{ 0,1\right\} ^{*}$
\item For every $n\in\mathbb{N}$, we define $T^{v_{n}}$ in the following
way: for the first $n$ levels of the tree, every node $a$ has only
one child $a0$. For the next $n-1$ levels, every node has both children.
Then again for the next $n$ levels of the tree every node $a$ has
only one child $a0$ and for the next $n-1$ levels of the tree we
let every node have both children. We repeat this process indefinitely.
\end{itemize}
Denote $F=\Gamma_{\Phi}\left(T\right)$. By the way $T$ is constructed,
for every $w\in W$, $\Gamma_{\Phi}\left(T^{w}\right)=\left[0,1\right]$,
and on the other hand, for every $n\in\mathbb{N}$, $\Gamma_{\Phi}\left(T^{v_{n}}\right)$
is a self-similar set with Hausdorff dimension $\dfrac{n-1}{2n-1}$.
By the way $T$ is constructed, for every limit $S=\lim\limits _{k\to\infty}T^{t_{k}}$
such that $t_{k}$ is not eventually constant, $\dimh\left(\Gamma_{\Phi}\left(S\right)\right)\in\left\{ 0,1\right\} $,
and so $\dfrac{1}{2}\notin\left\{ \dimh\left(S\right):\,S\in\overline{\mathcal{S}_{\Phi}^{T}}\right\} $.
However, $\dimh\left(\left[\dfrac{1}{3},1\right]\cap F\right)=\dfrac{1}{2}$.
\end{example}

Since the boundary of the unit cube $Q$ in $\mathbb{R}^{d}$ has
Hausdorff dimension $d-1$, minisets with Hausdorff dimension larger
than $d-1$ have to intersect the interior of $Q$ and their dimension
stays the same after removing the boundary of $Q$. Hence the following
holds.
\begin{prop}
\label{prop:minisets of branch sets in int(Q) vs branch sets dim>d-1}Let
$\Phi=\left\{ \varphi_{i}\right\} _{i\in\Lambda}$ be a similarity
IFS and $T\in\mathscr{T}_{\Lambda}^{\prime}$. Then
\[
\overline{\left\{ \dimh\left(B\right):B\text{ is a miniset of some}\,S\in\overline{\mathcal{S}_{\Phi}^{T}}\right\} }\cap\left(d-1,d\right]=\overline{\left\{ \dimh\left(S\right):\,S\in\overline{\mathcal{S}_{\Phi}^{T}}\right\} }\cap\left(d-1,d\right].
\]
\end{prop}

\subsection{Limits of descendant trees as microsets}
\begin{proof}[Proof of Theorem \ref{thm:minisets are microsets}]
The miniset $M$ has the form $M=Q\cap\psi\left(A\right)$, for some
$A\in\overline{\mathcal{S}_{\Phi}^{T}}$, and some similarity function
$\psi\left(x\right)=O\left(\lambda\cdot x\right)+r$. Since $A\in\overline{\mathcal{S}_{\Phi}^{T}}$,
there is some sequence of compact sets $A_{i}=\gamma_{\Phi}\left(\partial T^{v_{i}}\right)\in\mathcal{S}_{\Phi}^{T}$
where $v_{i}\in T$, s.t. $A_{i}\longrightarrow A$ in the Hausdorff
metric. This implies that $\psi\left(A_{i}\right)\longrightarrow\psi\left(A\right)$
as well. As in Lemma \ref{lem: convergence of a sequence and thickenings of a set},
denote $\delta_{n}=d_{H}\left(\psi\left(A_{i}\right),\psi\left(A\right)\right)$,
and notice that $Q^{\left(\delta_{n}\right)}\subset\psi_{n}\left(Q\right)$,
where $\psi_{n}:\mathbb{R}^{d}\to\mathbb{R}^{d}$ is the homothety
given by 
\[
\psi_{n}\left(x\right)=\left(1+2\delta_{n}\right)\cdot x-\delta_{n}\cdot\left(\begin{array}{c}
1\\
\vdots\\
1
\end{array}\right).
\]
This is the homothety that maps $Q$ to the smallest cube that contains
$Q^{\left(\delta_{n}\right)}$ (which is not a cube).

By Lemma \ref{lem: convergence of a sequence and thickenings of a set},
\[
\psi_{n}\left(Q\right)\cap\psi\left(A_{i}\right)\underset{n\to\infty}{\longrightarrow}Q\cap\psi\left(A\right),
\]
and since $\psi_{n}^{-1}\longrightarrow Id$, using Lemma \ref{lem:uniform convergence of functions preserve Hausdorff convergence},
one obtains 
\begin{equation}
Q\cap\psi_{n}^{-1}\circ\psi\left(A_{i}\right)=\psi_{n}^{-1}\left(\psi_{n}\left(Q\right)\cap\psi\left(A_{i}\right)\right)\underset{n\to\infty}{\longrightarrow}Q\cap\psi\left(A\right).\label{eq:convergence to a microset}
\end{equation}
Note that $A_{i}\subseteq\varphi_{v_{i}}^{-1}\left(F\right)$ for
every $i$, and therefore, 
\[
Q\cap\psi_{n}^{-1}\circ\psi\left(A_{i}\right)\subseteq Q\cap\psi_{n}^{-1}\circ\psi\circ\varphi_{v_{i}}^{-1}\left(F\right).
\]
Hence, any limit of the sequence $Q\cap\psi_{n}^{-1}\circ\psi\circ\varphi_{v_{i}}^{-1}\left(F\right)$
contains $Q\cap\psi\left(A\right)$. By compactness such a limit exists
and by definition it is a microset of $F$%
, which concludes the proof of the first part of the theorem.

Moreover, if $\psi^{-1}\left(Q\right)\subseteq U$ for an OSC set
$U$, since $\psi_{n}^{-1}\longrightarrow Id$, $\psi_{n}\circ\psi^{-1}\left(Q\right)\subseteq U$
whenever $n$ is large enough, and by moving to a subsequence we may
assume this is the case for all $n$. Since $U$ is an OSC set, %
\[
Q\cap\psi_{n}^{-1}\circ\psi\circ\varphi_{v_{i}}^{-1}\left(F\right)=Q\cap\psi_{n}^{-1}\circ\psi\left(A_{i}\right),
\]
and hence, by Equation (\ref{eq:convergence to a microset}), $M$
is in fact a microset of $F$.
\end{proof}
\begin{prop}
\label{prop:S contained in U as microsets}Let $F\subseteq\mathbb{R}^{d}$
be a non-empty compact set, and let $\Phi=\left\{ \varphi_{i}\right\} _{i\in\Lambda}$
be a similarity IFS which satisfies the OSC, such that $F=\gamma_{\Phi}\left(\partial T\right)$
for some tree $T\in$$\mathbb{\mathscr{T}}_{\Lambda}^{\prime}$. Then
for every $\text{\ensuremath{S\in}}\overline{\mathcal{S}_{\Phi}^{T}}$
such that $S\subseteq U$ for an OSC set $U$, there are finitely
many microsets $A_{1},\dots,A_{k}\in\mathcal{W}_{F}$, and similarity
maps $\eta_{1},\dots,\eta_{k}$, such that 
\begin{equation}
S=\bigcup\limits _{j=1}^{k}\eta_{j}\left(A_{j}\right).\label{eq:S is a union of contracted microsets}
\end{equation}
\end{prop}

\begin{proof}
By compactness, $S$ may be covered by finitely many closed cubes
$C_{1},\dots,C_{k}$ which are contained in $U$ and centered at points
of $S$. This implies that 
\[
S=\bigcup\limits _{j=1}^{k}C_{j}\cap S.
\]
For every $j\in\left\{ 1,\dots,k\right\} $, denote by $\eta_{j}$
a similarity function such that $\eta_{j}\left(Q\right)=C_{j}$, and
denote $A_{j}=Q\cap\eta_{j}^{-1}\left(S\right)$. We may assume that
each of the cubes $C_{1},\dots,C_{k}$ has side-length smaller than
1, and so the similarities $\eta_{j}^{-1}$ are expanding and the
sets $A_{j}$ are minisets of $S$. Since $C_{j}\subseteq U$ for
every $j$, by Theorem \ref{thm:minisets are microsets}, $A_{j}$
is a microset of $F$, and since each $C_{j}$ is centered at a point
of $S$, $A_{j}\in\mathcal{W}_{F}$.
\end{proof}
A version of Proposition \ref{prop:S contained in U as microsets}
also holds for minisets of $S\in\overline{\mathcal{S}_{\Phi}^{T}}$.
\begin{prop}
\label{prop:miniset of S contained in U as microsets}Let $F,\,\Phi,\,T$
be as in Proposition \ref{prop:S contained in U as microsets}, and
assume that for $S\in\overline{\mathcal{S}_{\Phi}^{T}}$ and a similarity
function $\psi$, $\psi^{-1}\left(Q\right)\cap S\subseteq U$ for
an OSC set $U$. Then the following holds 
\[
Q\cap\psi\left(S\right)=\bigcup\limits _{j=1}^{k}\chi_{j}\left(A_{j}\right),
\]
for microsets $A_{1},\dots,A_{k}$ of $F$, and similarity maps $\chi_{1},\dots,\chi_{k}$.
\end{prop}

\begin{proof}
The proof of this assertion is similar to the proof of Proposition
\ref{prop:S contained in U as microsets}, but $C_{1},...,C_{k}$
need to be chosen differently, so that they are all contained in $\psi^{-1}\left(Q\right)$
(and not necessarily centered at points of $S$). For example, by
partitioning the cube $\psi^{-1}\left(Q\right)$ to subcubes of diameter
smaller than $\text{dist}\left(\psi^{-1}\left(Q\right)\cap S,U^{c}\right)$,
and then taking $C_{1},...,C_{k}$ to be those subcubes that intersect
$S$ (and therefore contained in $U$). Hence,
\[
\psi^{-1}\left(Q\right)\cap S=\bigcup_{j=1}^{k}C_{j}\cap S.
\]
As in the proof of Proposition \ref{prop:S contained in U as microsets},
for every $j\in\left\{ 1,\dots,k\right\} $, denote by $\eta_{j}$
a similarity function such that $\eta_{j}\left(Q\right)=C_{j}$, and
$A_{j}=Q\cap\eta_{j}^{-1}\left(S\right)$. Applying Theorem \ref{thm:minisets are microsets},
each $A_{j}$ is a microset of $F$, and denoting $\chi_{j}=\psi\circ\eta_{j}$
for every $j$, 
\[
Q\cap\psi\left(S\right)=\psi\left(\bigcup\limits _{j=1}^{k}\eta_{j}\left(A_{j}\right)\right)=\bigcup\limits _{j=1}^{k}\chi_{j}\left(A_{j}\right).
\]
\end{proof}
Note that unlike Proposition \ref{prop:S contained in U as microsets},
in Proposition \ref{prop:miniset of S contained in U as microsets}
there is no guarantee that the microsets $A_{1},\dots A_{k}$ intersect
the interior of $Q$.

Whenever the IFS $\Phi$ satisfies the SSC, by Lemma \ref{lem:SSC implies OSC set containing the attractor}
there is an OSC set $U$ which contains the attractor of $\Phi$,
and therefore it also contains every $\text{\ensuremath{S\in}}\overline{\mathcal{S}_{\Phi}^{T}}$.
Hence, we obtain the following corollary:
\begin{cor}
\label{cor:dim of minisets contained in dim of microsets under SSC}Let
$F\subseteq\mathbb{R}^{d}$ be a non-empty compact set, and let $\Phi=\left\{ \varphi_{i}\right\} _{i\in\Lambda}$
be a similarity IFS which satisfies the SSC, such that $F=\Gamma_{\Phi}\left(T\right)$
for some tree $T\in$$\mathbb{\mathscr{T}}_{\Lambda}^{\prime}$, then
\begin{enumerate}
\item \label{enu:dim(S) subset of dim(microset)}$\left\{ \dimh\left(A\right):\,A\in\mathcal{W}_{F}\right\} \supseteq\left\{ \dimh\left(S\right):\,S\in\overline{\mathcal{S}_{\Phi}^{T}}\right\} $
\item \label{enu:dim of minisets =00003D dim of microsets where dim > d-1}$\left\{ \dimh\left(A\right):\,A\in\mathcal{W}_{F}\right\} \cap\left(d-1,d\right]=\left\{ \dimh\left(B\right):B\text{ is a miniset of some}\,S\in\overline{\mathcal{S}_{\Phi}^{T}}\right\} \cap\left(d-1,d\right]$
\end{enumerate}
\end{cor}

\begin{proof}
(\ref{enu:dim(S) subset of dim(microset)}) follows immediately from
Proposition \ref{prop:S contained in U as microsets}. To prove (\ref{enu:dim of minisets =00003D dim of microsets where dim > d-1}),
note that every miniset $B$ of some $S\in\overline{\mathcal{S}_{\Phi}^{T}}$
is equal to a finite union of minisets of $S$ that are contained
in $U$, and at least one of these minisets has the same Hausdorff
dimension as $B$. The rest is an immediate consequence of Proposition
\ref{prop:miniset of S contained in U as microsets} combined with
Corollary \ref{cor:microsets with dim>d-1} and with the observation
that if a microset has Hausdorff dimension larger than $d-1$, then
it has a non-empty intersection with $Q^{\circ}$.
\end{proof}
Note that Corollary \ref{cor:dim of minisets contained in dim of microsets under SSC}
above could be especially useful when $d=1$.

\subsection{Consequences regarding Assouad and lower dimensions}
\begin{lem}
\label{lem:Branch-set containd in a miniset}Let $\Phi=\left\{ \varphi_{i}\right\} _{i\in\Lambda}$
be a similarity IFS , and $S\in$$\mathbb{\mathscr{T}}_{\Lambda}^{\prime}$
some infinite tree. Then for every miniset $M\subseteq Q$ of $\Gamma_{\Phi}\left(S\right)$
s.t. $M\cap Q^{\circ}\neq\emptyset$, there is a non-empty set $F\in\mathcal{S}_{\Phi}^{S}$
such that $\psi\left(F\right)\subseteq M\cap Q^{\circ}$ for some
similarity function $\psi$. In particular $\dimh\left(F\right)\leq\dimh\left(M\right)$.
\end{lem}

\begin{proof}
Pick any $x\in M\cap Q^{\circ}$, and assume $\gamma_{\Phi}\left(i\right)=x$
for $i=i_{1}i_{2}\cdots\in\Lambda^{\mathbb{N}}$. For a large enough
$k\in\mathbb{N}$, 
\[
\Gamma_{\Phi}\left(\left[i_{1}\cdots i_{k}\right]\cap S\right)\subseteq M\cap Q^{\circ}.
\]
Denote $F=\Gamma_{\Phi}\left(S^{i_{1}\cdots i_{k}}\right)$, then
$\varphi_{i_{1}\cdots i_{k}}\left(F\right)=\Gamma_{\Phi}\left(\left[i_{1}\cdots i_{k}\right]\cap S\right)\subseteq M\cap Q^{\circ}$.
\end{proof}
As a consequence, the following Lemma follows at once:
\begin{lem}
\label{lem:inf of branch sets =00003D inf of minisets}Let $\Phi=\left\{ \varphi_{i}\right\} _{i\in\Lambda}$
be a similarity IFS , and $T\in$$\mathbb{\mathscr{T}}_{\Lambda}^{\prime}$
some infinite tree. Then 
\[
\inf\left\{ \dimh\left(S\right):\,S\in\overline{\mathcal{S}_{\Phi}^{T}}\right\} =\inf\left\{ \dimh\left(M\right):\,M\,\text{is a miniset of some }S\in\overline{\mathcal{S}_{\Phi}^{T}},\text{ and }M\cap Q^{\circ}\neq\emptyset\right\} ,
\]
and if one of the sets has a minimum, so does the other.
\end{lem}

\begin{proof}[Proof of Theorem \ref{thm:Assuad and lower dimensions as branching sets}]

(\ref{enu:dim_A}): By Theorem \ref{thm:Assouad as max}, 
\[
\dima\left(F\right)=\max\left\{ \dimh\left(A\right):\,A\in\mathcal{W}_{F}\right\} .
\]

Now, let $A\in\mathcal{W}_{F}$ be such a microset of maximal dimension.
By Theorem \ref{thm:microset is a union of minisets}, $A$ is contained
in a union of minisets of elements of $\overline{\mathcal{S}_{\Phi}^{T}}$,
hence, there is some $S\in\overline{\mathcal{S}_{\Phi}^{T}}$ with
$\dimh\left(A\right)\leq\dimh\left(S\right)$. On the other hand,
given any $S\in\overline{\mathcal{S}_{\Phi}^{T}}$, for some similarity
function $\psi$ with a scaling ratio > 1, $\dimh\left(Q\cap\psi\left(S\right)\right)=\dimh\left(S\right)$.
We may assume that $Q^{\circ}\cap\psi\left(S\right)\neq\emptyset$,
otherwise we may take $\psi$ with a slightly smaller scaling ratio.
By Theorem \ref{thm:minisets are microsets}, there is some microset
in $\mathcal{W}_{F}$ which contains $Q\cap\psi\left(S\right)$ and
therefore has a Hausdorff dimension $\geq\dimh\left(S\right)$.

(\ref{enu:dim_L_lower_bound}): Again, by Theorem \ref{thm:Assouad as max},
\[
\diml\left(F\right)=\min\left\{ \dimh\left(A\right):\,A\in\mathcal{W}_{F}\right\} .
\]
Given a microset $A\in\mathcal{W}_{F}$, by Theorem \ref{thm:microset is a union of minisets},
$Q^{\circ}\cap\psi\left(F\right)\subseteq A$ for some $F\in\overline{\mathcal{S}_{\Phi}^{T}}$
and an expanding similarity $\psi$, such that $Q^{\circ}\cap\psi\left(F\right)\neq\emptyset$.
Given a tree $Y\in\overline{\mathcal{D}_{T}}$ such that $F=\Gamma_{\Phi}\left(Y\right)$,
by Lemma \ref{lem:Branch-set containd in a miniset} there is a set
$S\in\mathcal{S}_{\Phi}^{Y}$ s.t. $\dim\left(S\right)\leq\dim\left(Q^{\circ}\cap\psi\left(F\right)\right)\leq\dim\left(A\right)$,
and by Lemma \ref{lem:branch set of a branch set}, $S\in\overline{\mathcal{S}_{\Phi}^{T}}$.

(\ref{enu:dim_L upper bound}): On the other hand, by Theorem \ref{thm:minisets are microsets},
a miniset $\psi\left(S\right)\cap Q$ of some $\text{\ensuremath{S\in}}\overline{\mathcal{S}_{\Phi}^{T}}$,
where $\psi^{-1}\left(Q\right)\subseteq U$ for an OSC set $U$, and
$\psi\left(S\right)\cap Q^{\circ}\neq\emptyset$, is a microset of
$F$, and therefore, the inequality follows from Theorem \ref{thm:Assouad as max}.

(\ref{enu:dim_L SSC}): Follows at once by combining (\ref{enu:dim_L_lower_bound})
and Corollary \ref{cor:dim of minisets contained in dim of microsets under SSC}.
\end{proof}
\begin{rem}
\label{rem:Smaller branch set than dim_L}Equality in (\ref{enu:dim_L_lower_bound})
of Theorem \ref{thm:Assuad and lower dimensions as branching sets}
need not hold. Consider for example the set $\left[\frac{1}{2},1\right]$
coded by binary expansion, i.e. by the IFS $\Phi=\left\{ f_{0}:x\mapsto\frac{1}{2}x,\,f_{1}:x\mapsto\frac{1}{2}x+\frac{1}{2}\right\} $.
One possible coding tree is given by $T$ where $T_{0}=\left\{ 0,1\right\} $,
$T^{1}=\left\{ 0,1\right\} ^{*}$ and $T^{0}=\left\{ \emptyset\right\} \cup\bigcup\limits _{n=1}^{\infty}\left\{ a_{1}\dots a_{n}\right\} $,
where $a_{1}a_{2}\dots=11111\dots$. Then w.r.t. the tree $T$, the
singleton $\left\{ 1\right\} \in\overline{\mathcal{S}_{\Phi}^{T}}$,
although $\diml\left[\frac{1}{2},1\right]=1$. The problem here is
the intersection of the cylinder sets. In case the IFS satisfies the
SSC, this problem goes away (as can be seen in (\ref{enu:dim_L SSC})
of Theorem \ref{thm:Assuad and lower dimensions as branching sets}).

Moreover, equality in (\ref{enu:dim_L upper bound}) also need not
hold. In fact, In some cases, $S\cap U=\emptyset$ for every $S\in\overline{\mathcal{S}_{\Phi}^{T}}$.
For example, take $F=\left[0,1\right]^{2}\setminus\left(0,1\right)^{2}$
coded using the IFS $\left\{ \varphi_{i}:\,\left(x,y\right)\mapsto\dfrac{1}{2}\left(x,y\right)+\dfrac{1}{2}i\right\} _{i\in\left\{ 0,1\right\} ^{2}}$
. Every OSC set $U$ for this IFS has to be contained in $\left(0,1\right)^{2}$,
and for any coding tree $T$ with respect to this IFS, for every $S\in\overline{\mathcal{S}_{\Phi}^{T}}$,
$S\cap U=\emptyset$.

\end{rem}

\section{Galton-Watson fractals\label{sec:Galton-Watson-fractals}}

\subsection{Basic definitions and results}

Although defined slightly differently, the name ``Galton-Watson
fractal'' as well as some of the formalism and notations regarding
trees are taken from the book \cite{Lyons2016}, which is a great
reference for an introduction to Galton-Watson processes.

In order to define Galton-Watson fractals, we first need to define
Galton-Watson trees.

Given a finite alphabet $\mathbb{A}$, and some random variable $W$
with values in $2^{\mathbb{A}}$, we assign for each word $a\in\mathbb{A}^{*}$
an independent copy of $W$ denoted by $W_{a}$. Then we define the
(random) Galton-Watson tree $T\in\mathscr{T}_{\mathbb{A}}$ by the
following process: we set $T_{0}=\left\{ \emptyset\right\} $, and
for every integer $n\geq1$, $T_{n}=\bigcup\limits _{a\in T_{n-1}}\left\{ aj:\,j\in W_{a}\right\} $.
Finally $T=\bigcup\limits _{n\geq0}T_{n}$ is the resulting Galton-Watson
tree. 

It may be easily seen that $T$ is indeed a tree in the alphabet $\mathbb{A}$.
In the case where $T$ is finite we say that the tree is \emph{extinct}
or that the event of \emph{extinction} occurred. Most of the events
considered in this paper are conditioned on non-extinction, i.e. on
the event that $T$ is infinite. It is well known that whenever $\mathbb{E}\left[\left|W\right|\right]>1$,
the probability of non-extinction is positive (see \cite{Lyons2016}
for more information). The case $\mathbb{E}\left[\left|W\right|\right]>1$
is called \emph{supercritical }and throughout the paper we only consider
this case.

By Kolomogorov's extension theorem, the process described above yields
a unique Borel measure on the space $\mathscr{T}_{\mathbb{A}}$, which
is the law of $T$, denoted here by $\mathscr{G}$. Since we only
deal with supercritical processes, the measure $\mathscr{G}$ may
be conditioned on the event of non-extinction. We denote this conditional
measure by $\mathscr{G}^{\infty}$ , which may be considered as a
measure on $\mathscr{T}_{\mathbb{A}}^{\infty}$. The notations $\mathscr{G},\,\mathscr{G}^{\prime}$
will be used throughout this subsection.

Given a similarity IFS $\Phi=\left\{ \varphi_{i}\right\} _{i\in\Lambda}$
with an attractor $K\subseteq\mathbb{R}^{d}$, one may construct a
GWT $T\in\mathscr{T}_{\Lambda}$. Assuming the process is supercritical,
there is a positive probability that $T\in\mathscr{T}_{\Lambda}^{\infty}$,
and in this case it may be projected to $\mathbb{R}^{d}$ using $\Gamma_{\Phi}$.
$E=\Gamma_{\Phi}\left(T\right)$ is then called a \emph{Galton-Watson
fractal} \emph{(GWF)}. Since the map $\Gamma_{\Phi}:\mathscr{T}_{\Lambda}^{\infty}\to\Omega_{K}$
is Borel (Proposition \ref{prop:Gamma_=00005CPhi is Borel}), the
probability distribution of $E$ is a Borel measure on $\Omega_{K}$,
which is obtained as the pushforward measure $\Gamma_{\Phi*}\left(\mathscr{G}^{\infty}\right)$.

As already mentioned, we only deal with similarity IFSs (although
it would be interesting to consider general affine IFSs). Whenever
the underlying similarity IFS satisfies the OSC, the Hausdorff dimension
of a Galton-Watson fractal is a.s. a fixed deterministic value which
may be easily calculated as stated in the following Theorem due to
Falconer \cite{FalconerK.J1986Rf} and Mauldin and Williams \cite{MauldinR.Daniel1986RRCA}.
One may also refer to \cite[Theorem 15.10]{Lyons2016} for another
elegant proof.
\begin{thm}
\label{thm:dimension of GW fractals}Let E be a Galton-Watson fractal
w.r.t. a similarity IFS $\Phi=\left\{ \varphi_{i}\right\} _{i\in\Lambda}$
satisfying the OSC, with contraction ratios $\left\{ r_{i}\right\} _{i\in\Lambda}$
and offspring distribution W. Then almost surely $\dimh E=\delta$
where $\delta$ is the unique number satisfying
\[
\mathbb{E}\left(\sum_{i\in W}r_{i}^{\delta}\right)=1.
\]
\end{thm}

The following theorem is a basic result in the theory of Galton-Watson
processes (\cite{Kesten1966}, see also \cite{Lyons2016}) which states
that in supercritical Galton-Watson processes, the size of the generations
grows at an exponential rate.
\begin{thm}[Kesten-Stigum]
\label{thm:Kesten - Stigum} Let $T$ be a supercritical GWT, then
$\dfrac{\left|T_{k}\right|}{\mathbb{E}\left[\left|W\right|\right]^{k}}$
converges a.s. (as $k\to\infty$) to a random variable $L$, where
$\mathbb{E}\left[L\right]=1$, and $L>0$ a.s. conditioned on non-extinction. 
\end{thm}

Another property which we shall use is stated in the following proposition,
which is a result of the inherent independence in the construction
of Galton-Watson trees (for a proof see \cite[Proposition 2.6]{dayan_2021})
\begin{prop}
\label{prop:0-1 law}Let $T$ be a supercritical Galton-Watson tree
with alphabet $\mathbb{A}$ and let $\mathscr{S}\subseteq\mathscr{T}_{\mathbb{A}}$
be a measurable subset. Suppose that $\mathbb{P}\left(T\in\mathscr{S}\right)>0$,
then a.s. conditioned on non-extinction, there exist infinitely many
$v\in T$ s.t. $T^{v}\in\mathscr{S}$.
\end{prop}

\subsection{Galton-Watson trees without leaves}

Recall that for a tree $S\in\mathscr{T}_{\mathbb{A}}^{\infty}$, $S^{\prime}$
denotes the reduced tree in $\mathscr{T}_{\mathbb{A}}^{\prime}$.
\begin{rem}
Note that for every nonempty $A\subseteq\mathbb{A}^{n}$ there is
a unique tree $F$ of height $n$ s.t. $F_{n}=A$ and $\left[F\right]\cap\mathscr{T}_{\mathbb{A}}^{\prime}\neq\emptyset$.
Hence, given a finite tree $F\in\mathscr{T}_{\mathbb{A}}$ of height
$n$ s.t. $\left[F\right]\cap\mathscr{T}_{\mathbb{A}}^{\prime}\neq\emptyset$,
for every $S\in\mathscr{\mathscr{T}_{\mathbb{A}}^{\prime}}$, $S\in\left[F\right]$
iff $S_{n}=F_{n}$. 
\end{rem}

\begin{thm}
\label{thm:induced tree is GWT}Let $T$ be a supercritical GWT on
an alphabet $\mathbb{A}$ with offspring distribution $W$ s.t. $\mathbb{P}\left[W=\emptyset\right]>0$.
Then conditioned on non-extinction, $T^{\prime}$ has the law of a
GWT with the offspring distribution $W^{\prime}$ defined as follows:
$\mathbb{P}\left[W^{\prime}=\emptyset\right]=0$, and for every $A\in2^{\mathbb{A}}\setminus\left\{ \emptyset\right\} $,

\[
\mathbb{P}\left[W^{\prime}=A\right]=\dfrac{1}{p}\sum_{A\subseteq B\subseteq\mathbb{A}}\mathbb{P}\left[W=B\right]\cdot p^{\left|A\right|}\cdot\left(1-p\right)^{\left|B\setminus A\right|}
\]
 where $p=\mathbb{P}\left[\text{non-extinction}\right]$.
\end{thm}

In a slightly different setting, relating only to the size of the
generations of $T^{\prime},$ it is shown in \cite[Theorem 5.28]{Lyons2016}
that $T^{\prime}$ is a Galton-Watson tree, which is described by
means of a probability generating function.

Obviously, in the case $\mathbb{P}\left[W=\emptyset\right]=0$, $T$
already has no leaves and so $T$ and $T^{\prime}$ coincide.

Note that indeed the normalization factor $p$ is correct. The expression
\begin{equation}
\sum_{A\subseteq2^{\mathbb{A}}\setminus\left\{ \emptyset\right\} }\sum_{A\subseteq B\subseteq\mathbb{A}}\mathbb{P}\left[W=B\right]\cdot p^{\left|A\right|}\cdot p^{\left|B\setminus A\right|}\label{eq:normalization factor for T'}
\end{equation}
may be interpreted as the probability that there exists some $A\subseteq2^{\mathbb{A}}\setminus\left\{ \emptyset\right\} $
s.t. $A\subseteq T_{1}$ and for every $v\in A$, $T^{v}$ is infinite,
and for every $w\in T_{1}\setminus A$, $T^{w}$ is finite. This is
exactly the probability that $T$ is infinite, i.e. the expression
in (\ref{eq:normalization factor for T'}) is equal to $p$. 

The following lemma will be used in order to prove Theorem \ref{thm:induced tree is GWT}.
\begin{lem}
Let $\tilde{T}$ be the GWT generated with the offspring distribution
$W^{\prime}$, then for every $A\subseteq\mathbb{A}^{n}\setminus\left\{ \emptyset\right\} $,
\[
\mathbb{P}\left[\tilde{T}_{n}=A\right]=\dfrac{1}{p}\sum_{A\subseteq B\subseteq\mathbb{A}^{n}}\mathbb{P}\left[T_{n}=B\right]\cdot p^{\left|A\right|}\cdot\left(1-p\right)^{\left|B\setminus A\right|}
\]
\end{lem}

\begin{proof}
The lemma is proved by induction on $n$. For $n=1$ the claim is
trivial. 

Assume this holds for $n-1$ and let $A\subseteq\mathbb{A}^{n}$.
Denote by $F$ the unique tree of height $n$ s.t. $F_{n}=A$ and
$\left[F\right]\cap\mathscr{T}_{\mathbb{A}}^{\prime}\neq\emptyset$.
Note that $\tilde{T}_{n}=A$ iff $\tilde{T}_{n-1}=F_{n-1}$ and for
every $v\in F_{n-1}$, $W_{\tilde{T}}\left(v\right)=W_{F}\left(v\right)$.
Hence, 
\[
\begin{array}{l}
\mathbb{P}\left[\tilde{T}_{n}=A\right]=\\
\mathbb{P}\left[\tilde{T}_{n-1}=F_{n-1}\right]\cdot\prod\limits _{v\in F_{n-1}}\mathbb{P}\left[W^{\prime}=W_{F}\left(v\right)\right]=\\
\left(\dfrac{1}{p}\sum\limits _{F_{n-1}\subseteq B}\mathbb{P}\left[T_{n}=B\right]\cdot p^{\left|F_{n-1}\right|}\cdot\left(1-p\right)^{\left|B\setminus F_{n-1}\right|}\right)\cdot\\
\left(\prod\limits _{v\in F_{n-1}}\dfrac{1}{p}\sum\limits _{W_{F}\left(v\right)\subseteq B}\mathbb{P}\left[W=B\right]\cdot p^{\left|W_{F}\left(v\right)\right|}\cdot\left(1-p\right)^{\left|B\setminus W_{F}\left(v\right)\right|}\right)=\\
\dfrac{1}{p}\cdot\underset{\Xi}{\underbrace{\left(\sum\limits _{F_{n-1}\subseteq B}\mathbb{P}\left[T_{n}=B\right]\cdot\left(1-p\right)^{\left|B\setminus F_{n-1}\right|}\right)}}\cdot\underset{\Upsilon}{\underbrace{\left(\prod\limits _{v\in F_{n-1}}\sum\limits _{W_{F}\left(v\right)\subseteq B}\mathbb{P}\left[W=B\right]\cdot p^{\left|W_{F}\left(v\right)\right|}\cdot\left(1-p\right)^{\left|B\setminus W_{F}\left(v\right)\right|}\right)}}
\end{array}
\]
Now, let us interpret the expressions $\Xi,\,\Upsilon$. $\Xi$ may
be interpreted as the probability of the following event: 
\[
F_{n-1}\subseteq T_{n-1}\text{ and }\forall v\in T_{n-1}\setminus F_{n-1},\,T^{v}\notin\mathscr{T}_{\mathbb{A}}^{\infty},
\]
and $\Upsilon$ may be interpreted as the probability of the event:
\[
\forall v\in F_{n-1},\,\left(W_{F}\left(v\right)\subseteq W_{v}\right)\text{ and }\left(\forall w\in W_{F}\left(v\right),\,T^{w}\in\mathscr{T}_{\mathbb{A}}^{\infty}\right)\text{ and }\left(\forall w\in W_{v}\setminus W_{F}\left(v\right),\,T^{w}\notin\mathscr{T}_{\mathbb{A}}^{\infty}\right).
\]
These events are independent and so the product of their probabilities
is the probability of their intersection. Their intersection is equivalent
to the following event:

\[
F_{n}\subseteq T_{n}\text{ and }\left(\forall w\in F_{n},\,T^{w}\in\mathscr{T}_{\mathbb{A}}^{\infty}\right)\text{ and }\left(\forall w\in T_{n}\setminus F_{n},\,T^{w}\notin\mathscr{T}_{\mathbb{A}}^{\infty}\right),
\]
whose probability is exactly 
\[
\dfrac{1}{p}\sum\limits _{A\subseteq B}\mathbb{P}\left[T_{n}=B\right]\cdot p^{\left|A\right|}\cdot\left(1-p\right)^{\left|B\setminus A\right|}.
\]
\end{proof}
\begin{proof}[proof of Theorem \ref{thm:induced tree is GWT}]

Let $\tilde{T}$ be a GWT generated according to the offspring distribution
$W^{\prime}$. We need to show that $\tilde{T}$ has the same law
as $T^{\prime}$ (conditioned on non-extinction). To do that, it is
enough to show that for every finite tree $F$ s.t. $\left[F\right]\cap\mathscr{T}_{\mathbb{A}}^{\prime}$,
\[
\mathbb{P}\left[T^{\prime}\in\left[F\right]\vert\,T\in\mathscr{T}_{\mathbb{A}}^{\infty}\right]=\mathbb{P}\left[\tilde{T}\in\left[F\right]\right].
\]
As mentioned above, $T^{\prime}\in\left[F\right]\iff T_{n}^{\prime}=F_{n}$
where $n$ is the height of $F$, and since $\tilde{T}$ has no leaves,
the same holds for $\tilde{T},$i.e. $\tilde{T}\in\left[F\right]\iff\tilde{T}_{n}=F_{n}$.
Therefore,

\[
\begin{array}{l}
\mathbb{P}\left[T^{\prime}\in\left[F\right]\vert\,T\in\mathscr{T}_{\mathbb{A}}^{\infty}\right]=\\
\mathbb{P}\left[T_{n}^{\prime}=F_{n}\vert\,T\in\mathscr{T}_{\mathbb{A}}^{\infty}\right]=\\
\mathbb{P}\left[T_{n}\supseteq F_{n}\text{ and }\left(\forall v\in F_{n},\,T^{v}\in\mathscr{T}_{\mathbb{A}}^{\infty}\right)\text{ and }\left(\forall v\in T_{n}\setminus F_{n},\,T^{v}\notin\mathscr{T}_{\mathbb{A}}^{\infty}\right)\vert\,T\in\mathscr{T}_{\mathbb{A}}^{\infty}\right]=\\
\dfrac{1}{p}\cdot\mathbb{P}\left[T_{n}\supseteq F_{n}\text{ and }\left(\forall v\in F_{n},\,T^{v}\in\mathscr{T}_{\mathbb{A}}^{\infty}\right)\text{ and }\left(\forall v\in T_{n}\setminus F_{n},\,T^{v}\notin\mathscr{T}_{\mathbb{A}}^{\infty}\right)\right]=\\
\dfrac{1}{p}\sum\limits _{F_{n}\subseteq B\subseteq\mathbb{A}}\mathbb{P}\left[T_{n}=B\right]\cdot p^{\left|F_{n}\right|}\cdot\left(1-p\right)^{\left|B\setminus F_{n}\right|}=\\
\mathbb{P}\left[\tilde{T}_{n}=F_{n}\right]=\mathbb{P}\left[\tilde{T}\in\left[F\right]\right]
\end{array}
\]
\end{proof}
Given a GWT $T$, as mentioned in the beginning of the section, its
probability distribution is denoted here by $\mathscr{G}$ and conditioned
on non-extinction by $\mathscr{G}^{\infty}$. From this point forward,
we also denote the probability distribution of $T^{\prime}$ (conditioned
on non-extinction) by $\mathscr{G}^{\prime}$. 

\subsection{The supports of the measures\label{subsec:The-supports-of-the-measures}}

Given any random variable $X$, we denote by $\supp\left(X\right)$
the support of its distribution. In what follows we often (but not
always) use this notation instead of addressing explicitly to the
probability distribution of the random variable.

Let $T$ be a GWT on the alphabet $\mathbb{A}$ with an offspring
distribution $W$. Keeping the notations introduced in the beginning
of the section, the probability distribution of $T$ is denoted by
$\mathscr{G}$. It may be easily seen that
\begin{equation}
\supp\left(\mathscr{G}\right)=\left\{ S\in\mathscr{T}_{\mathbb{A}}:\,\forall v\in S,\,W_{S}\left(v\right)\in\supp\left(W\right)\right\} ,\label{eq:support of T}
\end{equation}
 and
\begin{equation}
\supp\left(\mathscr{G}^{\infty}\right)=\supp\left(\mathscr{G}\right)\cap\mathscr{T}_{\mathbb{A}}^{\infty}.\label{eq:support of infinite T}
\end{equation}
Assuming that $\mathbb{P}\left[W=\emptyset\right]>0$, since according
to Theorem \ref{thm:induced tree is GWT} $T^{\prime}$ is itself
a GWT with an offspring distribution $W^{\prime}$, such that 
\[
\supp\left(W^{\prime}\right)=\left\{ B\in2^{\mathbb{A}}\setminus\left\{ \emptyset\right\} :\,\exists A\in\supp\left(W\right),\,B\subseteq A\right\} ,
\]
the support of $T^{\prime}$ is given by
\begin{equation}
\supp\left(\mathscr{G}^{\prime}\right)=\left\{ S\in\mathbb{\mathscr{T}}_{\mathbb{A}}^{\prime}:\,\forall v\in S,\,\exists A\in\supp\left(W\right),\,\emptyset\neq W_{S}\left(v\right)\subseteq A\right\} .\label{eq:support of T'}
\end{equation}

\begin{prop}
Let $T$ be a GWT on the alphabet $\mathbb{A}$ and with an offspring
distribution $W$. Then 
\[
\left\{ S^{\prime}:\,S\in\supp\left(\mathscr{G}^{\infty}\right)\right\} =\supp\left(\mathscr{G}^{\prime}\right).
\]
\end{prop}

\begin{proof}
In view of the Equations (\ref{eq:support of T}), (\ref{eq:support of infinite T}),
(\ref{eq:support of T'}), the inclusion $\subseteq$ is trivial.

For the other inclusion, assume $S\in\supp\left(T^{\prime}\right)$.
By Equation (\ref{eq:support of T'}), for every $v\in S$, $W_{S}\left(v\right)\subseteq B$
for some $B$ in $\supp\left(W\right)$. Choose for every $v\in S$
such a set $B_{v}$, and construct a new tree $P\supseteq S$ by
setting for every $v\in S$, $W_{P}\left(v\right)=B_{v}$ and for
every $w\in W_{P}\left(v\right)\setminus W_{S}\left(v\right)$, set
$W_{P}\left(w\right)=\emptyset$. Then clearly $P\in\supp\left(T\right)$
and $P^{\prime}=S$.

\end{proof}
Recall that according to Lemma \ref{lem:Projection of trees without leaves is continuous},
the map $\Gamma_{\Phi}\restriction_{\mathbb{\mathscr{T}}_{\Lambda}^{\prime}}:\mathbb{\mathscr{T}}_{\Lambda}^{\prime}\to\Omega_{\Phi}$
is continuous, and therefore %
\[
\Gamma_{\Phi}\left(\supp\left(\mathscr{G}^{\prime}\right)\right)=\supp\left(E\right).
\]
Combining this with the above Proposition, one obtains
\begin{lem}
\label{lem:supp(E)=00003Dprojection of supp(T^infty)}Let $E$ be
a GWF as above, then
\[
\supp\left(E\right)=\Gamma_{\Phi}\left(\supp\left(\mathscr{G}^{\infty}\right)\right)
\]
\end{lem}

Recall the notation $\mathscr{D}_{T}=\left\{ T^{v}:\,v\in T\right\} $
for a tree $T$. A direct consequence of Proposition \ref{prop:0-1 law}
is the following.
\begin{prop}
\label{prop:descendant trees are a.s. dense}Let $T$ be a supercritical
GWT on the alphabet $\mathbb{A}$, and let $\mathscr{G}$ be the corresponding
measure on $\mathscr{T}_{\mathbb{A}}$. Then a.s. conditioned on non-extinction,
\[
\overline{\mathscr{D}_{T}}=\supp\left(\mathscr{G}\right).
\]

\end{prop}

Since $T^{\prime}$ is also a GWT, Proposition \ref{prop:descendant trees are a.s. dense}
may be applied to $\mathscr{G}^{\prime}$, and so we obtain that a.s.
conditioned on non-extinction,
\[
\overline{\mathscr{D}_{T^{\prime}}}=\supp\left(\mathscr{G}^{\prime}\right),
\]
Now, given a similarity IFS $\Phi=\left\{ \varphi_{i}\right\} _{i\in\Lambda}$
and a tree $T\in\mathscr{T}_{\Lambda}^{\prime}$, recall the notation
$\mathcal{S}_{\Phi}^{T}=\Gamma_{\Phi}\left(\mathscr{D}_{T}\right)$,
and the fact that $\overline{\mathcal{S}_{\Phi}^{T}}=\overline{\Gamma_{\Phi}\left(\mathscr{D}_{T}\right)}=\Gamma_{\Phi}\left(\overline{\mathscr{D}_{T}}\right)$
(Equation (\ref{eq:change order of closure and Gamma_=00005CPhi})).
Summarizing the above, the following holds:
\begin{thm}
\label{thm:supp of GWF =00003D limit branch sets a.s.}Let $T$ be
a supercritical Galton-Watson tree w.r.t. a similarity IFS $\Phi$,
and with $E$, the corresponding Galton-Watson fractal. Then

\[
\supp\left(E\right)=\Gamma_{\Phi}\left(\supp\left(\mathscr{G}^{\prime}\right)\right)=\Gamma_{\Phi}\left(\supp\left(\mathscr{G}^{\infty}\right)\right),
\]
and a.s. conditioned on non-extinction,
\[
\supp\left(E\right)=\overline{\mathcal{S}_{\Phi}^{T^{\prime}}}.
\]
\end{thm}

Recall that $\supp\left(\mathscr{G}^{\prime}\right)$ was formulated
explicitly in (\ref{eq:support of T'}). 

Now that we have analyzed the set $\overline{\mathcal{S}_{\Phi}^{T^{\prime}}}$,
we shall apply the results of Section \ref{sec:Microsets-via-coding}.

\subsection{Trees with a prescribed set of allowed configurations}

Given a finite alphabet $\mathbb{A}$, and a collection $\mathscr{A}\subseteq2^{\mathbb{A}}$,
a tree $T\in\mathscr{T}_{\mathbb{A}}$ is called an $\mathscr{A}\text{-tree}$
if for every $v\in T$, $W_{T}\left(v\right)\in\mathscr{A}$.

For an IFS $\Phi=\left\{ \varphi_{i}\right\} _{i\in\Lambda}$, and
every non-empty set $A\subseteq\Lambda$, we denote the attractor
of the IFS $\left\{ \varphi_{i}\right\} _{i\in A}$ by $K_{A}$. Obviously,
$K_{A}\subseteq K$ where $K=K_{\Lambda}$ is the attractor of $\Phi$. 

In the rest of this subsection we are about to prove the following
Theorem:
\begin{thm}
\label{thm:dimension of A-trees is an interval}Given a similarity
IFS $\Phi=\left\{ \varphi_{i}\right\} _{i\in\Lambda}$ which satisfies
the OSC, and a non-empty collection $\mathscr{A}\subseteq2^{\Lambda}$,
denote 
\[
m_{\mathscr{A}}=\min\limits _{A\in\mathscr{A}}\left(\dimh\left(K_{A}\right)\right)\text{ and }M_{\mathscr{A}}=\max\limits _{A\in\mathscr{A}}\left(\dimh\left(K_{A}\right)\right),
\]
where in case $\emptyset\in\mathscr{A}$, we set $\dimh\left(K_{\emptyset}\right)=0$
(and so $m_{\mathscr{A}}=0$). Then the following equality holds:
\[
\left\{ \dimh\left(\Gamma_{\Phi}\left(T\right)\right):\,T\text{ is an infinite }\mathscr{A}\text{-tree}\right\} =\left[m_{\mathscr{A}},M_{\mathscr{A}}\right].
\]
\end{thm}

Note that it is enough to prove the Theorem for the case $\emptyset\notin\mathscr{A}$.
Indeed, assume that $\emptyset\in\mathscr{A}$. Denote $\underline{\mathscr{A}}=\left\{ A\in2^{\Lambda}\setminus\left\{ \emptyset\right\} :\,\exists B\in\mathscr{A},\,A\subseteq B\right\} $.
Then

\[
\left\{ T^{\prime}:\,T\in\mathscr{T}_{\Lambda}^{\infty}\,\text{is an }\mathscr{A}\text{-tree}\right\} =\left\{ T\in\mathscr{T}_{\Lambda}^{\infty}:\,T\text{ is an }\underline{\mathscr{A}}\text{-tree}\right\} .
\]
This implies that 
\[
\left\{ \dimh\left(\Gamma_{\Phi}\left(T\right)\right):\,T\text{ is an infinite }\mathscr{A}\text{-tree}\right\} =\left\{ \dimh\left(\Gamma_{\Phi}\left(T\right)\right):\,T\text{ is an infinite }\underline{\mathscr{A}}\text{-tree}\right\} .
\]
Now, $\emptyset\notin\underline{\mathscr{A}}$, however, unless $\mathscr{A}=\left\{ \emptyset\right\} $
(which we assume is not the case), $\underline{\mathscr{A}}$ contains
singletons, which implies that $m_{\underline{\mathscr{A}}}=m_{\mathscr{A}}=0.$
Of course, $M_{\underline{\mathscr{A}}}=M_{\mathscr{A}}$ as well,
and so assuming the claim holds for $\underline{\mathscr{A}}$ implies
that it holds for $\mathscr{A}$ as well. Therefore, in view of the
above, we assume for the rest of this subsection that $\emptyset\notin\mathscr{A}$.

In order to prove the theorem we first need to establish several lemmas.
\begin{lem}
\label{lem:dimension of self-similar set on section}For every $\mathscr{A}\text{-tree}$
$T$, and every section $\Pi\subset\Lambda^{*}$, 
\[
m_{\mathscr{A}}\leq\dimh\left(K_{T\cap\Pi}\right)\leq M_{\mathscr{A}}
\]
\end{lem}

\begin{proof}
Given a section $\Pi$, we denote $h\left(\Pi\right)=\max\limits _{i\in\Pi}\left|i\right|$.
The claim is proved by induction on $h\left(\Pi\right)$. 

When $h\left(\Pi\right)=1$, the claim follows immediately from the
definition of $m_{\mathscr{A}},M_{\mathscr{A}}$ as $T\cap\Pi\in\mathscr{A}$. 

Assume the claim holds for sections $\Pi^{\prime}$ with $h\left(\Pi^{\prime}\right)=n-1$,
and that $h\left(\Pi\right)=n$. Let $\Pi^{-}$ be the section defined
as follows:
\[
\Pi^{-}=\left(\Pi\setminus\Lambda^{n}\right)\cup\left\{ w\in\Lambda^{n-1}:\,\exists v\in\Pi,\,w<v\right\} 
\]
It may be easily seen that $\Pi^{-}$ is a section with $h\left(\Pi^{-}\right)=n-1$.
Recall that $\dimh\left(K_{T\cap\Pi^{-}}\right)=\alpha$ where $\alpha$
is the unique number satisfying
\[
\sum_{i\in T\cap\Pi^{-}}r_{i}^{\alpha}=1.
\]
By the induction hypothesis, $\sum\limits _{i\in T\cap\Pi^{-}}r_{i}^{m_{\mathscr{A}}}\geq1$
and $\sum\limits _{i\in T\cap\Pi^{-}}r_{i}^{M_{\mathscr{A}}}\leq1$.

For every $\alpha$, 
\begin{equation}
\sum\limits _{i\in T\cap\Pi}r_{i}^{\alpha}=\sum\limits _{i\in T\cap\Pi^{-}\cap\Pi}r_{i}^{\alpha}+\sum_{w\in T\cap\Pi^{-}\setminus\Pi}r_{w}^{\alpha}\cdot\sum_{v\in W_{T}\left(w\right)}r_{v}^{\alpha}.\label{eq:dimension of self-similar set on section}
\end{equation}
For every $w\in T$, $W_{T}\left(w\right)\in\mathscr{A}$, and therefore,
$\sum\limits _{v\in W_{T}\left(w\right)}r_{v}^{m_{\mathscr{A}}}\geq1$
and $\sum\limits _{v\in W_{T}\left(w\right)}r_{v}^{M_{\mathscr{A}}}\leq1$.
Hence, replacing $\alpha$ by $m_{\mathscr{A}}$ in Equation (\ref{eq:dimension of self-similar set on section}),
using the induction hypothesis we obtain $\sum\limits _{i\in T\cap\Pi}r_{i}^{m_{\mathscr{A}}}\geq1$
and replacing $\alpha$ by $M_{\mathscr{A}}$ we obtain $\sum\limits _{i\in T\cap\Pi}r_{i}^{M_{\mathscr{A}}}\leq1$
which concludes the proof.
\end{proof}
\begin{lem}
\label{lem:size of section}For any $\rho\in\left(0,r_{\min}\right)$,
the following holds for every $\mathscr{A}\text{-tree}$ $T$:
\[
\rho^{-m_{\mathscr{A}}}\leq\left|T_{\Pi_{\rho}}\right|\leq\rho^{-M_{\mathscr{A}}}\cdot r_{\min}^{-M_{\mathscr{A}}}.
\]
\end{lem}

\begin{proof}
Denote $\alpha=\dimh\left(K_{T\cap\Pi_{\rho}}\right)$, which means
that $\sum\limits _{i\in T\cap\Pi_{\rho}}r_{i}^{\alpha}=1$. Hence,
it follows from Equation (\ref{eq:bounds on contraction in sections})
that 
\[
\left|T_{\Pi_{\rho}}\right|\cdot\rho^{\alpha}r_{\min}^{\alpha}<1\leq\left|T_{\Pi_{\rho}}\right|\cdot\rho^{\alpha},
\]
or put differently, that 
\[
\rho^{-\alpha}\leq\left|T_{\Pi_{\rho}}\right|<\rho^{-\alpha}r_{\min}^{-\alpha}.
\]

By Lemma \ref{lem:dimension of self-similar set on section}, $m_{\mathscr{A}}\leq\alpha\leq M_{\mathscr{A}}$,
which implies the required inequalities.
\end{proof}
The following lemma is taken from \cite[Proposition 4.1]{dayan_2021}.
\begin{lem}
\label{lem:next gen in compressed tree}Given any $\rho\in\left(0,r_{\min}\right)$,
and any $n\in\mathbb{N}$, for every $i\in\Pi_{\rho^{n}}$ we denote
$a\left(i\right)=\dfrac{r_{i}}{\rho^{n}}$. Then for every $j\in\Lambda^{*}$,
\[
ij\in\Pi_{\rho^{n+1}}\iff j\in\Pi_{\frac{\rho}{a\left(i\right)}}.
\]
\end{lem}

\begin{lem}
\label{lem:number of children in compressed tree}As before, let $T$
be an $\mathscr{\ensuremath{A}}\text{-tree}$. Given any $\varepsilon>0$,
whenever $\rho\in\left(0,r_{\min}\right)$ is small enough, for every
$n\in\mathbb{N}$, and for every $i\in T\cap\Pi_{\rho^{n}}$, 
\[
\rho^{-\left(m_{\mathscr{A}}-\varepsilon\right)}\leq\left|\left\{ w\in T\cap\Pi_{\rho^{n+1}}:\,i<w\right\} \right|.
\]
\end{lem}

\begin{proof}
By Lemma \ref{lem:next gen in compressed tree}, for any $\rho\in\left(0,r_{\min}\right)$,
$n\in\mathbb{N}$, and $i\in T\cap\Pi_{\rho^{n}}$, 
\[
\left\{ w\in T\cap\Pi_{\rho^{n+1}}:\,i<w\right\} =\left\{ ij:\,j\in T^{i}\cap\Pi_{\frac{\rho}{a\left(i\right)}}\right\} .
\]
Since $T^{i}$ is also an $\mathscr{A}\text{-tree}$, by Lemma \ref{lem:size of section},
\[
\rho^{-m_{\mathscr{A}}}\cdot a\left(i\right)^{m_{\mathscr{A}}}\leq\left|T^{i}\cap\Pi_{\frac{\rho}{a\left(i\right)}}\right|.
\]
Note that $r_{\min}<a\left(i\right)\leq1$, hence, taking $\rho$
small enough so that $\rho^{\varepsilon}\leq r_{\min}^{m_{\mathscr{A}}}$,
we obtain
\[
\rho^{-m_{\mathscr{A}}+\varepsilon}\leq\left|T^{i}\cap\Pi_{\frac{\rho}{a\left(i\right)}}\right|
\]
for every $n\in\mathbb{N}$, and $i\in T\cap\Pi_{\rho^{n}}$, as claimed.
\end{proof}
We are now ready to prove Theorem \ref{thm:dimension of A-trees is an interval}.
\begin{proof}[Proof of Theorem \ref{thm:dimension of A-trees is an interval}]
(The inclusion $\subseteq$): For any $\varepsilon\in\left(0,m_{\mathscr{A}}\right)$,
choose some $\rho\in\left(0,r_{\min}\right)$ small enough for Lemma
\ref{lem:next gen in compressed tree} to hold for the given $\varepsilon$,
and such that $\rho^{-m_{\mathscr{A}}+\varepsilon}$ is an integer.
Conceptually, we think of the set $\left\{ \emptyset\right\} \cup\bigcup\limits _{n=1}^{\infty}T\cap\Pi_{\rho^{n}}$
as a tree whose $n^{\text{th}}$ level is $T\cap\Pi_{\rho^{n}}$,
and for every $i\in T\cap\Pi_{\rho^{n}}$, we think of the set $\left\{ v\in T\cap\Pi_{\rho^{n+1}}:\,i<v\right\} $
as the set of children of $i$. According to Lemma \ref{lem:number of children in compressed tree},
every such $i\in T\cap\Pi_{\rho^{n}}$ has at least $\rho^{-m_{\mathscr{A}}+\varepsilon}$
descendants in the section $T\cap\Pi_{\rho^{n+1}}$. By choosing for
each $i\in T\cap\Pi_{\rho^{n}}$ a set of exactly $\rho^{-m_{\mathscr{A}}+\varepsilon}$
descendants, we obtain a $\rho^{-m_{\mathscr{A}}+\varepsilon}\text{-ary}$
subtree, which projects using $\Gamma_{\Phi}$ to an Ahlfors-regular
subset of $\Gamma_{\Phi}\left(T\right)$ of Hausdorff dimension $m_{\mathscr{A}}-\varepsilon$
(see \cite[Lemma 4.21]{dayan_2021}). This implies, in particular,
that $\dimh\left(\Gamma_{\Phi}\left(T\right)\right)\geq m_{\mathscr{A}}-\varepsilon$.
Since $\varepsilon$ is arbitrary, we obtain the required lower bound.

On the other hand, By Lemma \ref{lem:size of section}, for every
$n\in\mathbb{N}$, 
\[
\left|T_{\Pi_{\rho^{n}}}\right|\leq\left(\rho^{n}\right)^{-M_{\mathscr{A}}}\cdot r_{\min}^{-M_{\mathscr{A}}},
\]
which is enough in order to determine that $\dimh\left(\Gamma_{\Phi}\left(T\right)\right)\leq M_{\mathscr{A}}$.

(The inclusion $\supseteq$): We prove this part using a probabilistic
argument. Recall that according to Theorem \ref{thm:dimension of GW fractals},
for a Galton-Watson fractal with an offspring distribution $\tilde{W}$,
the Hausdorff dimension is a.s. (conditioned on non-extinction) the
unique $s\geq0$ s.t. the following equality holds 
\begin{equation}
\mathbb{E}\left(\sum_{i\in\tilde{W}}r_{i}^{s}\right)=1.\label{eq:Hausdorff dimension of GWF inside proof}
\end{equation}
Denote by $A_{\min},A_{\max}\in\mathscr{A}$ any two sets that satisfy
\[
\dimh\left(K_{A_{\min}}\right)=m_{\mathscr{A}}\text{ and }\dimh\left(K_{A_{\max}}\right)=M_{\mathscr{A}}.
\]
Given any $t\in\left[0,1\right]$, consider a GWF which is constructed
using the offspring distribution $W_{t}=t\delta_{A_{\min}}+\left(1-t\right)\delta_{A_{\max}}$
and denote its a.s. Hausdorff dimension by $s_{t}$, i.e., $s_{t}$
is the solution of equation (\ref{eq:Hausdorff dimension of GWF inside proof})
with $W_{t}$ instead of $\tilde{W}$. Note that $\supp\left(W_{t}\right)\subseteq\supp\left(W\right)$.
Now, given any $a\in\left[m_{\mathscr{A}},M_{\mathscr{A}}\right]$,
since the mapping $t\mapsto s_{t}$ is continuous%
, there is some (unique) $t_{a}$ s.t. $s_{t_{a}}=a$. This means
that a GWF constructed with the offspring distribution $W_{t_{a}}$,
almost surely has its Hausdorff dimension equal to $a$, in particular
there exists an $\mathscr{A}\text{-tree}$ $T\in\mathscr{T}_{\Lambda}$,
such that $\Gamma_{\Phi}\left(T\right)=a$.
\end{proof}
\begin{rem}
The OSC was used quite extensively in the proof of Theorem \ref{thm:dimension of A-trees is an interval}.
However it might not be necessary for the conclusion of the Theorem
to hold and it would be interesting to try proving the theorem without
it.
\end{rem}

\subsection{Microsets and dimensions of Galton-Watson fractals\label{subsec:Microsets-and-dimensions-of GWF}}

We now apply Theorem \ref{thm:microset is a union of minisets} to
Galton-Watson fractals. Using Theorem \ref{thm:supp of GWF =00003D limit branch sets a.s.}
we obtain:
\begin{thm}
\label{thm:microsets as minisets for GWF}Let $E$ be a Galton-Watson
fractal w.r.t. a similarity IFS $\Phi$ which satisfies the WSC. Then
there exists some $n\geq1$, which depends only on $\Phi$, s.t. almost
surely for every microset $M$ of $E$, there are $A_{1},...,A_{n}\in\supp\left(E\right)$,
and expanding similarity maps $\psi_{1},...,\psi_{n}$, such that
\[
Q^{\circ}\cap\left(\bigcup_{j=1}^{n}\psi_{i}A_{i}\right)\subseteq M\subseteq Q\cap\left(\bigcup_{j=1}^{n}\psi_{i}A_{i}\right).
\]
Moreover, if $\Phi$ satisfies the SSC, then $n=1$.
\end{thm}

On the other hand, we apply Theorem \ref{thm:minisets are microsets}
and obtain the following:
\begin{thm}
\label{thm:minisets as microsets for GWF}Let $E$ be a Galton-Watson
fractal w.r.t. a similarity IFS $\Phi$. Then almost surely, every
miniset $M$ of a member of $\supp\left(E\right)$ is contained in
a microset of $E$. 

Moreover, if $M$ has the form $M=Q\cap\psi\left(A\right)$ for some
$A\in\supp\left(E\right)$ and an expanding similarity map $\psi$,
and assuming that $\Phi$ satisfies the OSC with an OSC set $U$ such
that $\psi^{-1}\left(Q\right)\subseteq U$, then $M$ is a microset
of $E$.
\end{thm}

In fact, when the underlying IFS satisfies the OSC, we can say more.
Here, we make use of the fact that the strong OSC is equivalent to
the OSC for similarity IFSs (Lemma \ref{lem:Separation conditions hierarchy}).

\begin{thm}
\label{thm:dimension of minisets in supp of GWF}Let $E$ be a Galton-Watson
fractal w.r.t. a similarity IFS $\Phi=\left\{ \varphi_{i}\right\} _{i\in\Lambda}$
which satisfies the OSC, and with offspring distribution $W$. Then
a.s. every miniset $M$ of a member of $\supp\left(E\right)$, is
contained in some microset $L$ of $E$, such that $\dimh\left(M\right)=\dimh\left(L\right)$.

Moreover, in the case where $\mathbb{P}\left[W=\emptyset\right]>0$,
a.s. every miniset of a member of $\supp\left(E\right)$ is a microset
of $E$.
\end{thm}

\begin{proof}
Let $T$ be the corresponding GWT. Suppose we are given an $M=Q\cap\psi\left(S\right)$
for some $S\in\supp\left(E\right)$ and an expanding similarity map
$\psi$. Recall that by Theorem \ref{thm:supp of GWF =00003D limit branch sets a.s.},
$\supp\left(E\right)=\Gamma_{\Phi}\left(\supp\left(T^{\prime}\right)\right)$.
Let $Y\in\supp\left(T^{\prime}\right)$ satisfy $\Gamma_{\Phi}\left(Y\right)=S$
. By the strong OSC there exists some $x\in K\cap U$, where $U$
is an OSC set. Assume that $x=\gamma_{\Phi}\left(i\right)$ for $i=i_{1}i_{2}\cdots\in\Lambda^{\mathbb{N}}$.
For $k$ large enough, $\varphi_{i_{1}\cdots i_{k}}\left(\psi^{-1}\left(Q\right)\right)\subseteq U$.
Since we assume that $\forall a\in\Lambda,\,\mathbb{P}\left[a\in W\right]>0$,
there is some finite tree $F$ of height $k$, s.t. $j:=i_{1}\cdots i_{k}\in F$,
and $\mathbb{P}\left[T^{\prime}\in\left[F\right]\right]>0$. Consider
the tree $\tilde{Y}\in\mathscr{T}_{\Lambda}^{\prime}$ constructed
as follows:
\begin{itemize}
\item $\tilde{Y}\in\left[F\right]$,
\item $\tilde{Y}^{j}=Y$,
\item For every $v\in F_{k}\setminus\left\{ j\right\} $, $\tilde{Y}^{v}=X$
where $X\in\supp\left(T^{\prime}\right)$ and $\dimh\left(\Gamma_{\Phi}\left(X\right)\leq\dimh\left(M\right)\right)$.
The existence of such $X$ is guaranteed by Lemma \ref{lem:inf of branch sets =00003D inf of minisets}
combined with Theorem \ref{thm:supp of GWF =00003D limit branch sets a.s.}.
\end{itemize}
It may be easily verified that $\tilde{Y}\in\supp\left(T^{\prime}\right)$
and therefore $\tilde{S}=\Gamma_{\Phi}\left(\tilde{Y}\right)\in\supp\left(E\right)$.

Denote $L:=Q\cap\psi\circ\varphi_{j}^{-1}\left(\tilde{S}\right)$.
Since $\varphi_{j}\circ\psi^{-1}\left(Q\right)\subseteq U$, by Theorem
\ref{thm:minisets as microsets for GWF}, $L$ is a microset of $E$%
. Note that $\tilde{S}=\bigcup\limits _{v\in F_{n}}\varphi_{v}\circ\Gamma_{\Phi}\left(\tilde{Y}^{v}\right)$,
and therefore

\[
L=\bigcup\limits _{v\in F_{n}}Q\cap\psi\circ\varphi_{j}^{-1}\circ\varphi_{v}\circ\Gamma_{\Phi}\left(\tilde{Y}^{v}\right).
\]
In particular, since $j\in F_{n}$, $M\subseteq L$. Together with
the fact that $\dim\left(\Gamma_{\Phi}\left(\tilde{Y}^{v}\right)\right)\leq\dimh\left(M\right)$
for every $v\in F_{n}\setminus\left\{ j\right\} $, this implies that
$\dimh\left(M\right)=\dimh\left(L\right).$

For the ``moreover'' part, in case $\mathbb{P}\left[W=\emptyset\right]>0$,
we could define $\tilde{Y}$ just by 
\[
\tilde{Y}=\left\{ \emptyset\right\} \cup\bigcup\limits _{n=1}^{k}\left\{ i_{1}\cdots i_{n}\right\} \cup jY
\]
 and proceed in the same way. Thus we obtain that $L=M$ is a microset.
\end{proof}

The following lemma may be proved using the same argument used in
the proof of Theorem \ref{thm:dimension of minisets in supp of GWF}.
To avoid redundancy, the proof of the lemma is omitted, and the required
adjustments are left to the reader.
\begin{lem}
\label{lem:sets in supp(E) as microsets of E using strong OSC }Let
$E$ be a Galton-Watson fractal w.r.t. a similarity IFS $\Phi=\left\{ \varphi_{i}\right\} _{i\in\Lambda}$,
which satisfies the OSC. Then for every $A\in\supp\left(E\right)$,
there exists some $\tilde{A}\in\supp\left(E\right)$ and an expanding
similarity $\psi$, such that the following hold:
\begin{enumerate}
\item $\psi^{-1}\left(Q\right)\subseteq U$ for an OSC set $U$;
\item $\dimh\left(A\right)=\dimh\left(\psi\left(\tilde{A}\right)\cap Q\right)$;
\item $\psi\left(\tilde{A}\right)\cap Q^{\circ}\neq\emptyset$.
\end{enumerate}
\end{lem}

Note that by Theorem \ref{thm:minisets as microsets for GWF}, a.s.
$\psi\left(\tilde{A}\right)\cap Q\in\mathcal{W}_{E}$.
\begin{proof}[Proof of Theorem \ref{thm:Hausdorff dimensions of microsets of GWF}]
Let $T$ be the corresponding GWT. As mentioned in Equation (\ref{eq:support of T}),
the support of $T$ is just the set of all $\supp\left(W\right)\text{-trees}$.
By Theorem \ref{thm:dimension of A-trees is an interval}, we have
\[
\left\{ \dimh\left(\Gamma_{\Phi}\left(S\right)\right):\,S\in\supp\left(\mathscr{G}^{\infty}\right)\right\} =\left[m_{W},M_{W}\right].
\]
Hence, by Theorem \ref{thm:supp of GWF =00003D limit branch sets a.s.}
we obtain 
\[
\left\{ \dimh\left(F\right):\,F\in\supp\left(E\right)\right\} =\left[m_{W},M_{W}\right].
\]

Combining Lemma \ref{lem:sets in supp(E) as microsets of E using strong OSC }
and Theorem \ref{thm:minisets as microsets for GWF} yields that a.s.
for every $A\in\supp\left(E\right)$ there is some $F\in\mathcal{W}_{E}$,
such that $\dimh\left(A\right)=\dimh\left(F\right)$. This shows that
$\left\{ \dimh\left(F\right):\,F\in\mathcal{W}_{E}\right\} \supseteq\left[m_{W},M_{W}\right]$. 

Recalling that a.s. conditioned on non-extinction, $\overline{\mathcal{S}_{\Phi}^{T^{\prime}}}=\supp\left(E\right)$
(Theorem \ref{thm:supp of GWF =00003D limit branch sets a.s.}), the
opposite inclusion follows from Corollary \ref{thm:Assuad and lower dimensions as branching sets}
combined with Theorem \ref{thm:Assouad as max}.
\end{proof}

\section{Some remarks on Furstenberg's definition of microsets and homogeneity.\label{sec:Fursternberg}}

We now recall Furstenberg's original definition of a microset. To
distinguish this definition from Definition \ref{def:microsets},
we use the terminology \emph{F-microset} and \emph{F-miniset}.
\begin{defn}
\label{def:F-microsets}Given a compact set $F\subseteq\mathbb{R}^{d}$,
we define the following:
\begin{enumerate}
\item A non-empty subset of a miniset of $F$, whose defining similarity
map is a homothety is called an F-miniset of $F$.
\item An F-microset of $F$ is any limit of F-minisets of $F$ in the Hausdorff
metric.
\end{enumerate}
We denote by $\mathcal{F}_{F}$ the set of all F-microsets of $F$.

The following Lemma establishes the relation between microsets as
defined in Definition \ref{def:microsets} and F-microsets as defined
above.
\end{defn}

\begin{lem}
\label{lem:F-microsets are subsets of microsets}Let $F\subseteq\rd$
and $M\subseteq Q$ be non-empty compact sets. Then $M$ is an F-microset
of $F$ $\iff$ $M$ is a subset of a microset of $F$ which is obtained
using only homotheties.
\end{lem}

\begin{proof}
($\implies$): Assume $M$ is an F-microset of $F$, then it is a
limit of a sequence $M_{i}$ of F-minisets of $F$. Each $M_{i}$
is contained in a miniset $L_{i}$ of $F$, which is obtained using
a homothety. By moving to a subsequence, $L_{i}$ converges in the
Hausdorff metric to some compact set $L$ which is by definition a
microset of $F$. Since $M_{i}\subseteq L_{i}$ for every $i$, $M\subseteq L$.

($\impliedby$): Assume $M\subseteq L$ where $L=\lim\limits _{i\to\infty}Q\cap\psi_{i}\left(F\right)$,
and all the $\psi_{i}$ are homotheties. By Lemma \ref{lem: convergence of a sequence and thickenings of a set},
there is a sequence $M_{i}$ of thickenings of $M$ such that 
\[
Q\cap\psi_{i}\left(F\right)\cap M_{i}\underset{i\to\infty}{\longrightarrow}L\cap M=M.
\]
Since each $Q\cap\psi_{i}\left(F\right)\cap M_{i}$ is an F-miniset
of $F$, $M$ is indeed an F-microset of $F$.
\end{proof}
Using the same arguments as in the proof of Theorem \ref{thm:microset is a union of minisets},
we may obtain the following:
\begin{thm}
\label{thm:Furstenberg microset is a union of minisets}Let $F\subseteq\rd$
be a non-empty compact set. Given an IFS $\Phi=\left\{ \varphi_{i}\right\} _{i\in\Lambda}$
of homotheties which satisfies WSC and a tree $T\in$$\mathbb{\mathscr{T}}_{\Lambda}^{\prime}$
such that $F=\Gamma_{\Phi}\left(T\right)$, there exists some $n\geq1$
which depends only on $\Phi$, such that every $M\in\mathcal{F}_{F}$
satisfies $M=\bigcup\limits _{j=1}^{n}S_{j}$, where $S_{1},\dots,S_{n}$
are F-minisets of some $A_{1},...,A_{n}\in\overline{\mathcal{S}_{\Phi}^{T}}$.
Moreover, if $\Phi$ satisfies the SSC, then $n=1$.
\end{thm}

On the other hand, the following theorem is an immediate consequence
of Theorem \ref{thm:minisets are microsets} together with Lemma \ref{lem:F-microsets are subsets of microsets}.
\begin{thm}
\label{thm:F-minisets are F-microsets}Let $F\subseteq\rd$ be a non-empty
compact set. Given an IFS $\Phi=\left\{ \varphi_{i}\right\} _{i\in\Lambda}$
of homotheties and a tree $T\in$$\mathbb{\mathscr{T}}_{\Lambda}^{\prime}$
such that $F=\Gamma_{\Phi}\left(T\right)$, every F-miniset of a member
of $\overline{\mathcal{S}_{\Phi}^{T}}$ is an F-microset of $F$.%
\end{thm}

Theorems \ref{thm:Furstenberg microset is a union of minisets} and
\ref{thm:F-minisets are F-microsets} are suitable versions of Theorems
\ref{thm:microset is a union of minisets} and \ref{thm:minisets are microsets}
for F-microsets. Since the definition of F-microsets allows taking
subsets, Theorems \ref{thm:Furstenberg microset is a union of minisets}
and \ref{thm:F-minisets are F-microsets} take nicer forms than their
parallels. In case the SSC holds, the following corollary is immediate:
\begin{cor}
\label{cor:F-microsets are F-minisets of S, for SSC}Given an IFS
$\Phi=\left\{ \varphi_{i}\right\} _{i\in\Lambda}$ of homotheties,
and a tree $T\in$$\mathbb{\mathscr{T}}_{\Lambda}^{\prime}$, denote
$F=\gamma_{\Phi}\left(\partial T\right)$. Assume that $\Phi$ satisfies
the SSC, then 
\[
\mathcal{F}_{F}=\left\{ M:\,M\text{ is an F-miniset of \ensuremath{S}},\,S\in\overline{\mathcal{S}_{\Phi}^{T}}\right\} .
\]
\end{cor}

The following definition is from \cite{Furstenberg2008405}.
\begin{defn}
A set $F$ is called \emph{homogeneous} if every F-microset of $F$
is an F-miniset of $F$.
\end{defn}

In \cite{Furstenberg2008405}, Furstenberg showed that every linear
map restricted to a homogeneous set has a property called dimension
conservation, whose definition will now be recalled.
\begin{defn}
A linear map $f:\mathbb{R}^{d}\to\mathbb{R}^{l}$ is \emph{dimension
conserving} on a compact set $A\subseteq\rd$ if there exists some
$\delta\geq0$, such that 
\[
\delta+\dimh\left\{ y\in\mathbb{R}^{l}:\,\dimh\left(f^{-1}\left(y\right)\cap A\right)\geq\delta\right\} \geq\dimh\left(A\right),
\]
with the convention that $\dimh\left(\emptyset\right)=-\infty$.
\end{defn}

Roughly speaking, $f:\rd\to\mathbb{R}^{l}$ is dimension conserving
on $A$ if any decrease of the Hausdorff dimension between $A$ and
$f\left(A\right)$ is compensated by the Hausdorff dimension of the
fibers $f^{-1}\left(y\right)$. This may be seen as a generalization
of the property described by the rank-nullity theorem for linear maps
between two vector spaces.

Now, assume that $F$ is the attractor of an IFS $\Phi=\left\{ \varphi_{i}\right\} _{i\in\Lambda}$
of homotheties, so that $T=\Lambda^{*}$ is a coding tree for $F$,
and $\overline{\mathcal{S}_{\Phi}^{T}}=\left\{ F\right\} $. If $\Phi$
satisfies the SSC, then by Corollary \ref{cor:F-microsets are F-minisets of S, for SSC},
$F$ is homogeneous. This is already well known. In \cite{KaenmakiRossi2016}
it was shown using a result from \cite{Fraser2015assouad} that for
self-homothetic sets in $\mathbb{R}$ with Hausdorff dimension $<1$,
homogeneity implies the WSC. However, it was shown that the WSC does
not characterize homogeneity, even for this class of sets, and there
exist such sets which satisfy the OSC, but are not homogeneous. In
that paper it is mentioned that a characterization for homogeneity
is not yet known. 

However, assuming that $\Phi$ satisfies the WSC, Theorem \ref{thm:Furstenberg microset is a union of minisets}
implies that every F-microset of $F$ is a union of some (bounded)
finite number of F-minisets of $F$. One may consider this property
as a weaker form of homogeneity.
\begin{defn}
\label{def:weak homogeneity}A set $F$ is called \emph{weakly homogeneous}
if every F-microset of $F$ is a union of finitely many F-minisets
of $F$.
\end{defn}

An immediate consequence of Theorem \ref{thm:Furstenberg microset is a union of minisets}
is the following:
\begin{prop}
\label{prop:WSC implies weak homogeneous}Let $F\subseteq\mathbb{R}^{d}$
be the attractor of an IFS of homotheties which satisfies the WSC.
Then $F$ is weakly homogeneous.
\end{prop}

Now, using the same arguments as in \cite{KaenmakiRossi2016} as well
as Proposition \ref{prop:WSC implies weak homogeneous}, we obtain
the following equivalence:
\begin{thm}
\label{thm:weak homogeneity iff WSC}Let $F\subseteq\mathbb{R}$ be
the attractor of $\Phi$ - an IFS of homotheties. Assume that $\dimh\left(F\right)<1$
and that $F$ is not a singleton. Then $F$ is weakly homogeneous
$\iff$ $\Phi$ satisfies WSC.
\end{thm}

\begin{proof}
The direction $\left(\impliedby\right)$ follows from Proposition
\ref{prop:WSC implies weak homogeneous}. For the direction $\left(\implies\right)$,
assume that $F$ is weakly homogeneous. By \cite[Propositions 5.7, 5.8]{Kaenmaki2018rigidity},
there is an F-microset $M$ of $F$ s.t. $\dimh\left(M\right)=\dima\left(F\right)$.
By weak homogeneity $\dimh\left(M\right)\leq\dimh\left(F\right)$
which implies that $\dima\left(F\right)<1$. By \cite[Theorem 3.1]{Fraser2015assouad},
this implies that $\Phi$ satisfies WSC.
\end{proof}
\begin{rem}
Some remarks on the Theorem:
\begin{enumerate}
\item For a self-similar set $F\subseteq\mathbb{R}$ which is not a singleton
and with $\dimh\left(F\right)<1$, the WSC of its defining IFS is
a property of the set $F$ itself. That is, if $F$ is the attractor
of some similarity IFS which satisfies the WSC, then every similarity
IFS whose attractor is $F$ satisfies the WSC as well. This follows
from \cite[Theorem 1.3]{Fraser2015assouad}.
\item Using the example constructed in \cite{KaenmakiRossi2016} of a self-homothetic
set which is the attractor of an IFS that satisfies the open set condition,
but fails to be homogeneous, it follows that weak homogeneity does
not imply homogeneity, and so these two properties are not equivalent,
and homogeneity is indeed a stronger property.
\end{enumerate}
\end{rem}

The weak homogeneity property also makes sense in the context of Furstenberg's
work since as we now show, weak homogeneity implies the same conclusion
regarding dimension conservation as homogeneity.
\begin{thm}
\label{thm:weak homogeneity implies dimension conservation}If a compact
set $A\subseteq\rd$ is weakly homogeneous, then every linear map
$f:\rd\to\mathbb{R}^{l}$ is dimension conserving on $A$.
\end{thm}

\begin{proof}
By \cite[Theorem 6.1]{Furstenberg2008405} $A$ has an F-microset
$M$ such that $\dimh\left(M\right)\geq$ $\dimh\left(A\right)$,
and $f$ is dimension conserving on $M$. Assuming that $A$ is weakly
homogeneous, $M=\bigcup\limits _{i=1}^{k}M_{i}$, where each $M_{i}$
is an F-miniset of $A$. 

$f$ being dimension conserving on $M$, means that for some $\delta\geq0$,
\[
\delta+\dimh\left\{ y:\,\dimh\left(f^{-1}\left(y\right)\cap M\right)\geq\delta\right\} \geq\dimh\left(M\right)\geq\dimh\left(A\right).
\]
Denote $L=\left\{ y:\,\dimh\left(f^{-1}\left(y\right)\cap M\right)\geq\delta\right\} $,
and for every $i$, denote $L_{i}=\left\{ y:\,\dimh\left(f^{-1}\left(y\right)\cap M_{i}\right)\geq\delta\right\} $.
Note that $L=\bigcup\limits _{i=1}^{k}L_{i}$, and therefore there
exists some $i_{0}\in\left\{ 1,...,k\right\} $ such that $\dimh\left(L_{i_{0}}\right)=\dimh\left(L\right)$.
Hence, we obtain 
\[
\delta+\dimh\left\{ y:\,\dimh\left(f^{-1}\left(y\right)\right)\cap M_{i_{0}}\geq\delta\right\} \geq\dimh\left(A\right).
\]

Now, $M_{i_{0}}$ is a miniset of $A$, so $M_{i_{0}}\subseteq\psi\left(A\right)$%
{} for some homothety $\psi:\rd\to\rd$. Therefore, we also have
\[
\delta+\dimh\left\{ y:\,\dimh\left(f^{-1}\left(y\right)\right)\cap\psi\left(A\right)\geq\delta\right\} \geq\dimh\left(\psi\left(A\right)\right).
\]

In words, this means that $f$ is dimension conserving on $\psi\left(A\right)$,
which is equivalent to $f$ being dimension conserving on $A$.
\end{proof}
\begin{cor}
Let $F\subseteq\rd$ be the attractor of an IFS of homotheties which
satisfies the WSC, then every linear map $f:\rd\to\mathbb{R}^{l}$
is dimension conserving on $F$.
\end{cor}

\begin{rem}
The weak homogeneity property could be given two alternative definitions.
One is weaker than Definition \ref{def:weak homogeneity}, requiring
that every F-microset is the union of at most countably many F-minisets.
The other is stronger than Definition \ref{def:weak homogeneity},
requiring that there exists some $c\in\mathbb{N}$ s.t. every F-microset
is the union of at most $c$ F-minisets.

In both cases, Theorem \ref{thm:weak homogeneity implies dimension conservation},
Proposition \ref{prop:WSC implies weak homogeneous} and Theorem \ref{thm:weak homogeneity iff WSC}
would still hold. Also, note that the latter implies that for self-homothetic
subsets of $\mathbb{R}$ with Hausdorff dimension <1, all these alternative
definitions are in fact equivalent.
\end{rem}

Finally, we consider the application of Theorem \ref{thm:Furstenberg microset is a union of minisets}
to Galton-Watson fractals. This yields the following probabilistic
version of the implication $\left(\impliedby\right)$ in Theorem \ref{thm:weak homogeneity iff WSC}.
\begin{thm}
\label{thm:weak homogeneity GWF}Let $E$ be a GWF, constructed with
respect to an IFS $\Phi$ of homotheties which satisfies the WSC,
then a.s. there exists some $n\in\mathbb{N}$ which depends only on
$\Phi$, such that every $M\in\mathcal{F}_{E}$ satisfies $M=\bigcup\limits _{j=1}^{n}S_{j}$,
where $S_{1},\dots,S_{n}$ are $F\text{-minisets}$ of some $A_{1},...,A_{n}\in\supp\left(E\right)$.
Moreover, if $\Phi$ satisfies the SSC, then $n=1$.
\end{thm}

The property that every F-microset of $E$ is a finite union of F-minisets
of elements of $\supp\left(E\right)$, may be seen as a generalization
of the weak homogeneity property to random fractals. Also note that
when $\Phi$ satisfies the SSC, we obtain the assertion that every
F-microset of $E$ is an F-miniset of some set in $\supp\left(E\right)$
which, again, may be considered as a generalization of the homogeneity
property to random fractals.

 \hypersetup{bookmarksdepth=-1}

\specialsection*{\textbf{Acknowledgments}}

The author wishes to thank Jonathan Fraser for some helpful tips,
and especially for the reference to his great book \cite{Fraser2020}.
The author is also thankful to Barak Weiss for some helpful comments
on the paper.\hypersetup{bookmarksdepth=2}

\bibliographystyle{abbrv}
\bibliography{all}

\end{document}